\documentclass[11pt]{amsart}
\usepackage{authblk}
\usepackage{stmaryrd}
\usepackage{amsmath}
\usepackage{amsfonts}
\usepackage{amssymb}
\usepackage{amsthm}
\usepackage{mathabx}
\usepackage[all,2cell]{xy}
\usepackage{color}
\usepackage{etoolbox}
\usepackage{tikz-cd}
\usepackage{relsize}
\usepackage{enumitem}
\UseTwocells

\newbool{hidedetails}
\booltrue{hidedetails}

\newtheorem{proposition}{Proposition}
\newtheorem{corollary}{Corollary}
\newtheorem{theorem}{Theorem}
\newtheorem*{propo}{Proposition}
\newtheorem*{theo}{Theorem}

\theoremstyle{definition}
\newtheorem{definition}{Definition}
\newtheorem{example}{Example}

\newtheorem{remark}{Remark}

\newcommand{\blue}{\color{blue}}

\newcommand{\hide}[1]{\ifbool{hidedetails}{}{{\blue #1}}}

\newcommand{\wh}[1]{{\widehat{#1}}}						
\newcommand{\ul}[1]{{\underline{#1}}}						
\newcommand{\wt}[1]{{\widetilde{#1}}}						
\newcommand{\hocolim}{{\rm hocolim}}						
\newcommand{\holim}{{\rm holim}}							
\newcommand{\homology}[1]{{H^{#1}}}						

\newcommand{\pt}{{{\rm p}}}								
\newcommand{\mani}{{\mathcal M}}							

\newcommand{\dersch}{{\mathbb S}}							
\newcommand{\unspa}{{\boldsymbol S}}						
\newcommand{\zeropart}[1]{{\mathcal #1}}					
\newcommand{\unspam}{{\boldsymbol M}}					
\newcommand{\opes}{{U}}									
\newcommand{\derm}{{\mathbb M}}							
\newcommand{\dera}[1]{{{#1}^\bullet}}						
\newcommand{\deri}[3]{{#1^{#3}_{#2}}}						
\newcommand{\spec}[1]{{{\rm Spec}(#1)}}						
\newcommand{\undera}[1]{{{#1}^*}}							
\newcommand{\mapp}{{\rm Map}}							

\newcommand{\cring}[2]{{{\mathcal #1}^{#2}}}					
\newcommand{\cringi}[3]{{{\mathcal #1}^{#3}_{#2}}}				
\newcommand{\cinfty}{{C^\infty}}							

\newcommand{\rahm}[2]{{\Omega^{#2}(#1)}}				
\newcommand{\cotan}[1]{{\mathbb L_{#1}}}					
\newcommand{\tang}[2]{{\mathbb T_{#1/#2}}}					
\newcommand{\tanga}[1]{{\mathbb T_{#1}}}					
\newcommand{\tangde}[2]{{\mathbb T_{#1}^{#2}}}				
\newcommand{\derham}[3]{{\Omega^{#3}({#1}/{#2})}}			
\newcommand{\drham}[3]{{\Omega^{#1}_{#2}(#3)}}				
\newcommand{\anchor}{{\alpha}}							
\newcommand{\anchord}{{\boldsymbol\alpha}}							
\newcommand{\dima}[1]{{\boldsymbol\mu_{#1}}}							
\newcommand{\lalge}{{\mathcal L}}							
\newcommand{\homke}[1]{{{\mathcal K}_{#1}}}					
\newcommand{\dedi}{{\rm d_{dR}}}							
\newcommand{\forms}[1]{{\overline{\boldsymbol\Omega}_{#1}}}			
\newcommand{\quotient}[2]{{#1/#2}}							
\newcommand{\reduced}[1]{{(#1)_{\rm red}}}					
\newcommand{\derhamspace}[1]{{(#1)_{\rm dR}}}				

\newcommand{\nform}[1]{{\omega_{#1}}}						
\newcommand{\symfo}{{\underline{\omega}}}					
\newcommand{\symstra}{{\boldsymbol\omega}}					
\newcommand{\refo}[1]{{{#1}_{\rm re}}}						
\newcommand{\imfo}[1]{{{#1}_{\rm im}}}						
\newcommand{\conju}[1]{{\overline{#1}}}						
\newcommand{\lagra}{{\underline{\lambda}}}					
\newcommand{\lagras}{{\boldsymbol\lambda}}					
\newcommand{\lagfo}[1]{{\lambda_{#1}}}						
\newcommand{\lmin}{{\mathfrak L}}							
\newcommand{\para}{{t}}									
\newcommand{\negacy}[2]{{{\rm NC^{#1}_{#2}}}}				
\newcommand{\symsh}[1]{{{\mathfrak S}_{#1}}}					
\newcommand{\ish}[1]{{{\mathfrak I}_{#1}}}					

\newcommand{\promo}{{P}}								
\newcommand{\dumo}[1]{{(#1)^\vee}}							
\newcommand{\grami}[2]{{\Lambda^{#2,#1}}}					
\newcommand{\gramii}[3]{{\Lambda_{#3}^{#2,#1}}}				
\newcommand{\pullback}[2]{{#1^*(#2)}}						
\newcommand{\conne}[1]{{\gamma_{#1}}}						
\newcommand{\makebo}[3]{{{\mathbf #1_{\mathbf #3}}}}							
\newcommand{\homdi}[1]{{\delta}}							
\newcommand{\cirdi}{{\epsilon}}							
\newcommand{\Sym}[3]{{{\rm Sym}_{#3}^{#1}(#2)}}				
\newcommand{\suspense}[1]{{#1[1]}}						
\newcommand{\suspensi}[2]{{#1[#2]}}						
\renewcommand{\hom}[1]{{{\rm Hom}_{#1}}}					
\newcommand{\tosplit}[1]{{ #1}}								
\newcommand{\lhalf}[1]{{#1'}}								
\newcommand{\rhalf}[1]{{#1''}}								
\newcommand{\nesplit}[1]{{{ #1}'}}							
\newcommand{\negsplit}[2]{{{#1}'_{#2}}}						
\newcommand{\posplit}[1]{{{#1}^+}}							
\newcommand{\negta}[1]{{(#1)^-}}							
\newcommand{\negtat}[2]{{(#1)^-_{#2}}}						
\newcommand{\ubun}{{ F}}								
\newcommand{\rk}[2]{{{\rm rk}_{#2}(#1)}}						
\newcommand{\bundle}[3]{{#1^{#2}_{#3}}}						

\tikzcdset{%
	arrow style=tikz,%
	diagrams={>={To[length=4.4pt,width=3.6pt,line width=0.4pt]}}%
}

\begin{document}

\title[Strictification and gluing of Lagrangian distributions]{Strictification and gluing of Lagrangian distributions on derived schemes with shifted symplectic forms}
\maketitle

\author{Dennis Borisov, Ludmil Katzarkov, Artan Sheshmani, Shing-Tung Yau}
{Dennis Borisov${}^{1}$, Ludmil Katzarkov${}^{4, 6,7}$, Artan Sheshmani${}^{2,3,4}$ and  Shing-Tung Yau$^{2,5}$}

\address{${}^1$  Department of Mathematics and Statistics, University of Windsor, 401 Sunset Ave, Windsor Ontario, Canada}

\address{${}^2$ Harvard University CMSA and Physics Department, Jefferson Laboratory, 17 Oxford St, Cambridge, MA 02138}
\address{${}^3$ Institut for Matematik , Aarhus Universitet, Ny Munkegade 118 Building 1530, DK-8000 Aarhus C, Denmark}
\address{${}^4$ National Research University Higher School of Economics, Russian Federation, Laboratory of Mirror Symmetry, NRU HSE, 6 Usacheva str.,Moscow, Russia, 119048}
\address{${}^5$ Department of Mathematics, Harvard University, Cambridge, MA 02138, USA}
\address{${}^6$ University of Miami, Coral Gables, FL}
\address{${}^7$ Institute of Mathematics and Informatics, Bulgarian Academy of Sciences}


\date{\today}

\begin{abstract} A strictification result is proved for isotropic distributions on derived schemes equipped with negatively shifted homotopically closed $2$-forms. It is shown that any derived scheme over $\mathbb C$ equipped with a $-2$-shifted symplectic structure, and having a Hausdorff space of classical points, admits a globally defined Lagrangian distribution as a dg $\cinfty$-manifold.

\smallskip



\end{abstract}

\tableofcontents

\section*{Introduction}

\subsection*{Background} Shifted symplectic structures were introduced by Pantev, To\"en, Vaqui\'e and Vezzosi in \cite{PTVV13}, and there it was proved that such structures always appear on moduli spaces of sheaves on Calabi--Yau manifolds. Briefly a shifted symplectic structure is the homotopy algebraic generalization of the usual notion of a symplectic form, where one imposes the two conditions of being de Rham closed and non-degenerate only up to homotopy. We recall some of this theory in Section \ref{SectionSymplecticStructures}.

A major step in understanding local structure of negatively shifted symplectic forms was made by Brav, Bussi and Joyce in \cite{BBJ13}. They have proved a local Darboux theorem, which states that locally any negatively shifted symplectic form can be strictified. This result implies, in particular, that any derived scheme carrying a $-1$-shifted symplectic form can locally be written as a derived critical locus of a function.

\smallskip

Just like the notion of a symplectic form, also Lagrangian distributions have a homotopy algebraic generalization. The property of being isotropic becomes a structure, which allows one to formulate the Lagrangian condition (\cite {PTVV13} \S 2.2). A systematic study of distributions in the context of derived geometry is being done in \cite {AlgF1}, \cite{AlgF2}. In \cite{P14} it was observed that presence of a Lagrangian distribution allows one to write the derived stack as a critical locus of a shifted potential, with the shift being the shift of the symplectic form plus 1. 

A strictified shifted symplectic form, i.e.\@ written in terms of Darboux coordinates, has a natural Lagrangian distribution, thus the local Darboux theorem of \cite{BBJ13} is equivalent to existence of local shifted potentials for all negatively shifted symplectic forms. In fact the main theorem of \cite{BBJ13} is proved by constructing local shifted potentials, and then producing Darboux coordinates.

 If one wants to construct Lagrangian distributions globally, and hence produce globally defined shifted potentials, one needs to develop a more flexible notion of strictification of Lagrangian distributions.
 
 \subsection*{Main results} Our goal is to prove strictification and gluing results for isotropic and Lagrangian distributions on derived schemes equipped with negatively shifted homotopically closed and symplectic forms. To do this we need to restrict to a particular class of derived distributions: firstly we require the distributions to admit local models that are given in terms of sub-complexes of the tangent complexes (this is what we call \emph{derived foliations} Def.\@ \ref{DefDerivedFoliation}), secondly we require that cohomologically the distributions are trivial in non-positive degrees (this is what we call \emph{purely derived} distributions  Def.\@ \ref{DefPurelyDerived}).
 
 Combining Prop.\@ \ref{StriPro} and Prop.\@ \ref{TrivLa} we have our main strictification result.
 
 \begin{propo} Let $\dera{A}$ be a dg algebra and let $\symstra$ be a homotopically closed $2$-form on $\spec{\dera{A}}$. Let $\anchor\colon(\lalge,\lagras)\longrightarrow(\tanga{{\dera{A}}},\symstra)$ be a purely derived foliation with an isotropic structure, s.t.\@ 
 	\begin{equation*}\homology{\geq 2}(\lalge)\overset\cong\longrightarrow\homology{\geq 2}(\tanga{{\dera{A}}}).\end{equation*} 
 	Then locally on $\spec{\dera{A}}$ there is an equivalent distribution with an isotropic structure $\anchor'\colon(\lmin,\lagras')\rightarrow(\tanga{\dera{A}},\symstra)$, s.t.\@ $\lmin$ is a perfect dg $\dera{A}$-module and for every $\pt\colon\dera{A}\rightarrow\mathbb C$ \begin{enumerate}[label=(\roman*)]
 		\item $\lmin^{\leq 0}|_\pt=0$,
 		\item $\lmin^{\geq 2}|_\pt\overset{\cong}\longrightarrow\tangde{{\dera{A}}}{\geq 2}|_\pt$,
 		\item $\lmin^1|_\pt\longrightarrow\tangde{{\dera{A}}}{1}|_{\pt}$ is injective.\end{enumerate}
 If, in addition, $\symstra$ is of degree $-2$ and $\forall\pt\colon\dera{A}\rightarrow\mathbb C$ it defines a perfect pairing between $\homology{0}\left(\tanga{\deri{A}{}{\bullet}}\right)|_\pt$ and $\homology{2}\left(\tanga{\deri{A}{}{\bullet}}\right)|_\pt$, then it can be arranged that $\symfo|_{\lmin}=0$ modulo $\homdi{}(\deri{A}{}{-1})$, where $\symfo$ is the leading term of $\symstra$.
 	\end{propo}
 Notice that this proposition deals with homotopically closed forms, not necessarily symplectic ones. This is because we would like to apply this strictification result to derived stacks equipped with shifted symplectic forms, e.g.\@ quotients of derived schemes by actions of groups. Then the symplectic structure exists only on the quotient stacks, while on the derived schemes one has just homotopically closed forms.
 
 \smallskip
 
 It is important that our strictification result applies to all isotropic distributions satisfying some conditions that are invariant with respect to weak equivalences. In particular, we can transport such distributions from one chart to another and then strictify. This leads to our main gluing result (Thm.\@ \ref{MainResult}).
 
 \begin{theo} Let $(\dersch,\symstra)$ be a derived scheme with a $\mathbb C$-valued $-2$-shifted symplectic structure. Let $(\derm,\imfo{\symstra},\refo{\symstra})$ be the underlying dg manifold with $\mathbb R$-valued symplectic structures. Suppose that the space $\mani$ of classical points is Hausdorff and second countable. Then the sheaf on $\mani$ of purely derived foliations that are Lagrangian distributions with respect to $\imfo{\symstra}$ and are negative definite with respect to $\refo{\symstra}$ is soft. In particular the set of global sections is not empty.\end{theo}

\subsection*{Applications} Although at least some of our statements are proved in greater generality, and many of them are useful for arbitrary negative shifts, it is the particular example of the moduli space of sheaves on a Calabi--Yau $4$-fold, that is the actual object of our study. As usual the motivation comes from the moduli spaces of bundles on Calabi--Yau $4$-folds, and much of what we do is mimicking the gauge-theoretic constructions in derived algebraic geometry, having the Uhlenbeck--Yau theorem in mind \cite{UY86}. The idea is due to Dominic Joyce, and was partially implemented in \cite{BoJ13}.

\smallskip

Recall that Donaldson and Thomas have suggested in \cite{DT} to use a choice of a top holomorphic form on a Calabi--Yau $4$-fold as a sort of orientation to define anti-self dual instantons. Briefly this means discarding some of the conditions on the curvature of a connection, so as to obtain a determined elliptic system. In \cite{Conan} this was shown by Leung to be part of a pattern of (special) connections, depending on the metric division algebra chosen. For Calabi--Yau $4$-folds all $4$ choices are possible, and the suggestion of \cite{DT} corresponds to the interaction between complex numbers and octonions. The case of Calabi--Yau $4$-folds is exceptional for being the only class of manifolds with octonion metric structure. 

Even more exceptional is the fact that on the set-theoretical level the moduli spaces of special $SU(4)$-connections (geometry over the complex numbers) and $Spin(7)$-connections (geometry over the octonions) are isomorphic. In terms of derived geometry, the latter moduli space, coming from a determined elliptic system, admits a virtual fundamental class. On the other hand, set-theoretic equality between the two moduli spaces allows us to view the first moduli space as a derived critical locus of a function of degree $-1$ on the second moduli space. 

\smallskip

Dominic Joyce had come up with the idea of replicating this remarkable correspondence in the setting of algebraic geometry, rather than gauge theory, and in the process obtaining natural compactifications of the moduli spaces. The Uhlenbeck--Yau theorem allows us to reformulate $SU(4)$-connections algebraic-geometrically, but there is no such luck for $Spin(7)$-instantons. The idea of Dominic was to arrive at the $Spin(7)$-moduli space on the algebraic geometric side by using the same pairing given by the top holomorphic form. This is how the $-2$-shifted symplectic structure becomes central.

In \cite{BoJ13} the question of gluing such local constructions was addressed. The objective was to arrive at the virtual fundamental class of $Spin(7)$-instantons in the algebraic-geometric setting. This was done by gluing local constructions up to cohomology. Different choices of isotropic/negative definite sub-bundles and the cohomological freedom of gluing local charts should not alter the resulting virtual fundamental class, i.e.\@ the cobordism class of the quotient derived manifold.

\smallskip

Here we have a slightly different point of view. Instead of striving for a virtual fundamental class, we are more interested in the correspondence between the moduli space of perfect sheaves (i.e.\@ $SU(4)$-connections in gauge-theoretic terms) and the moduli space of the algebraic-geometric version of $Spin(7)$-instantons on the same Calabi--Yau manifold. In terms of derived algebraic geometry this means a correspondence between the total space and the base space of a Lagrangian fibration. Here an observation by Tony Pantev, made in \cite{P14}, becomes important.

He observed that, given a Lagrangian distribution on a derived scheme or a derived stack equipped with a shifted symplectic structure, one can divide by this distribution obtaining another scheme/stack {together with a globally defined graded potential}. The critical locus of this potential is the original scheme/stack. If the distribution happens to be the tangent distribution to a morphism, and, moreover, the domain and codomain of this morphism have isomorphic reduced schemes, then the codomain will be the quotient of the domain by the distribution.

On the gauge-theoretic side equality of the reduced schemes is precisely the equality of the sets of $SU(4)$-connections and $Spin(7)$-connections. The same remains true also for the larger moduli spaces on the side of derived algebraic geometry. Finally the local constructions in \cite{BoJ13} exhibit the Lagrangian distributions as tangent distributions. Therefore our objective is to show that such Lagrangian distributions exist globally on the moduli stack of perfect complexes on any Calabi--Yau $4$-fold.

\smallskip

Here we need to take into account the stacky nature of our moduli spaces. It is not clear to us how to proceed in the abstract general case of an arbitrary Artin stack with a $-2$-shifted symplectic structure on it. The special case of the moduli spaces of perfect complexes on Calabi--Yau $4$-folds has important simplifying properties: we can choose our moduli stack to be the quotient of a derived $Quot$-scheme. Locally this quotient map is a $PGL(n)$-principal bundle. Thus one can work on the derived $Quot$-scheme itself, using the group action to distribute the construction along the orbits. This is done in \cite{BSY2}.

\subsection*{Previous results} Similarly to \cite{BoJ13} our approach is by gluing local constructions. However, both the local objects and the gluing method are different. Locally the object of our interest is a Lagrangian distribution, which consists of an integrable distribution and an isotropic structure, that satisfies the homological analog of the maximality condition. In \cite{BoJ13} only the distributions were constructed, the isotropic structures were always assumed to be trivial. This difference becomes important when we try to extend a Lagrangian distribution from one local chart to another.

In general a shifted symplectic structure is given on each local chart individually, together with gluing up to homotopy on the intersections. Therefore it might be difficult to find two distributions with trivial isotropic structures on two charts, that agree on the intersection. Allowing non-trivial isotropic structures and, moreover, keeping track of them is essential for the gluing.

As in \cite{BoJ13} the gluing is made possible by using partition of unity. As in the loc.\@ cit.\@ the key to this is the requirement that the distribution is not only isotropic with respect to the imaginary part of the symplectic form, but also negative definite with respect to the real part. This approach works for \emph{strict} Lagrangian distributions, i.e.\@ distributions where the imaginary part of the symplectic form vanishes on the nose. This requirement of strictness was the main source of technical difficulties in \cite{BoJ13}, that eventually led to gluing up to cohomology.

In this paper we do not always require our Lagrangian distributions to be strict, imposing only cohomological conditions instead. This immediately implies that a pull-back of such a distribution from one chart to another still has this property. To perform the gluing argument we need to prove a strictification result, and here the isotropic structure becomes essential. We prove that, using the isotropic structure, every Lagrangian distribution that we consider, can be equivalently rewritten as a strict Lagrangian distribution. 

This utilization of isotropic structures and the accompanying strictification result are the main difference between our approach and that of \cite{BoJ13}. Together they allow us to prove existence of a genuine globally defined Lagrangian distribution. Another minor difference in our approaches is in using the sheaf of all possible Lagrangian distributions instead of a single atlas. Explicitly, on the topological space of classical points in the moduli space we define the sheaf of Lagrangian distributions relative a given symplectic form, and then translate existence of partition of unity into softness of this sheaf, which, given the local constructions, immediately implies existence of global sections.

\subsection*{Contents of the paper} In Section \ref{IntegrableDistributions} we recall the infinitesimal calculus in the context of derived geometry and the theory of integrable distributions in terms of graded mixed algebras, as described in \cite{P14}. Almost everything in this section is well known, the only exception being the notion of a derived foliation (Def.\@ \ref{DefDerivedFoliation}). Roughly a derived foliation is a distribution, s.t.\@ locally it always has a presentation as a Lie sub-algebra of the tangent complex. This requirement implies a non-trivial cohomological condition, that considerably simplifies our constructions in the rest of the paper.

In Section \ref{SectionSymplecticStructures} we take care of symplectic forms, isotropic distributions and Lagrangian distributions. We recall the necessary definitions and results from \cite{PTVV13}, \cite{P14} and describe how each one of the three kinds of structure can be organized into a sheaf on the topological space of classical points. We show that the sets of global sections of these sheaves agree with the usual definitions of globally defined structures.

The main constructions of the paper begin in Section \ref{SectionStrict}. There we introduce the notion of a purely derived foliation, which is a derived foliation that has vanishing cohomology in non-positive degrees. For distribution of this kind we prove the semi-strictification result: there is an equivalent rewriting where the symplectic form vanishes on the nose, but only modulo the ideal of the classical scheme. It is here that the isotropic structure is used extensively. In addition we show that for purely derived foliations, if at least one isotropic structure is Lagrangian, so are all of them, i.e.\@ it is a property of the distribution itself.

Almost everything in sections \ref{IntegrableDistributions}, \ref{SectionSymplecticStructures}, \ref{SectionStrict} applies equally well to dg algebras over $\mathbb C$ and to dg $\cinfty$-rings. Section \ref{SectionGlobalLagrangian} deals only with dg $\cinfty$-rings. Here we finish proving the strictification result, i.e.\@ that every Lagrangian distribution that we consider can be made strict, and use this to show softness of the sheaf of such distributions on the topological space of classical points.

\subsection*{Notation} By \emph{a classical point} in a derived scheme we mean a morphism from $\spec{\mathbb C}$ to this scheme. Similarly for derived manifolds over $\mathbb R$. Given a dg ring $\dera{A}$ we write \emph{perfect dg $\dera{A}$-module} to mean a retract of a finitely generated almost free dg $\dera{A}$-module. For such dg module we can choose a finite sequence of finitely generated projective $\deri{A}{}{0}$-submodules, that generate it over $\dera{A}$. Sometimes we will call these $\deri{A}{}{0}$-submodules \emph{generating bundles}.

\medskip

{\bf Acknowledgements:} \includegraphics[scale=0.2]{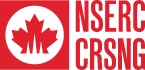} The first author acknowledges the support of the Natural Sciences and Engineering Research Council of Canada (NSERC), [RGPIN-2020-04845].

Cette recherche a été financée par le Conseil de recherches en sciences naturelles et en génie du Canada (CRSNG), [RGPIN-2020-04845].

\smallskip

The first author would like to thank Dominic Joyce, Tony Pantev and Dingxin Zhang for very helpful conversations. The second author was partially supported by NSF Grant, Simons Investigatior  Award HMS, Simons Collaboration Award HMS, National Science Fund of Bulgaria, National Scientific Program “Excellent Research and People for the Development of European
Science” (VIHREN), Project No. KP-06-DV-7, HSE University Basic Research Program.The third author would like to thank Dennis Gaitsgory, Amin Gholampour, Martijn Kool, Naichung Conan Leung, and Tony Pantev for many helpful discussions and commenting on the first versions of this article. Research of A.S. and D. B. was partially supported by generous Aarhus startup research grant of A. S., and partially by the NSF DMS-1607871, NSF DMS-1306313, the Simons 38558, and Laboratory of Mirror Symmetry NRU HSE, RF Government grant, ag. No 14.641.31.0001. A.S. would like to further sincerely thank the Center for Mathematical Sciences and Applications at Harvard University, the center for Quantum Geometry of Moduli Spaces at Aarhus University, and the Laboratory of Mirror Symmetry in Higher School of Economics, Russian federation, for the great help and support. S.-T. Y. was partially supported by NSF DMS-0804454,
NSF PHY-1306313, and Simons 38558. 

\section{Integrable distributions}\label{IntegrableDistributions}

In this section we recall the notion of integrable distributions on derived schemes and the process of dividing by such distributions. We follow \cite{P14} and use the language of graded mixed algebras. We will be especially interested in the cases when the result of such division is again a derived scheme. We begin by recalling the well known facts concerning cotangent complexes and complexes of K\"ahler differentials.

\subsection{Well presented dg algebras and $\cinfty$-rings}

Much of what we will be doing will involve de Rham complexes corresponding to non-positively graded differential algebras over $\mathbb C$. It is well known of course that, given such an algebra $\dera{A}=(\undera{A},\homdi{})$, one should derive the functor of K\"ahler differentials to obtain \emph{the cotangent complex} $\cotan{{\dera{A}}}$. 

For practical reasons one would like to know when can one use the na\"ive K\"ahler differentials and still obtain the correct derived version. The following statement, which the authors have learned from Chris Brav, provides a useful class of dg algebras having this property.

\begin{proposition}\label{ComputingCotangent} Let $\dera{A}=(\undera{A},\homdi{})$ be a differential non-positively graded $\mathbb C$-algebra, s.t.\@ $A^0$ is of finite type and smooth over $\mathbb C$ and $\undera{A}$ is freely generated as an $A^0$-algebra by a sequence $\{\promo^k\}_{k<0}$ of finitely generated projective $A^0$-modules, where $\forall k<0$ $\promo^k\subseteq A^k$. Then the dg $\dera{A}$-module $\drham{}{}{\dera{A}}$ of K\"ahler differentials has the homotopy type of $\cotan{{\dera{A}}}$.\end{proposition}
\begin{proof} Let $\widetilde{A}^\bullet\twoheadrightarrow A^\bullet$ be an almost free resolution of $A^\bullet$, i.e. $\widetilde{A}^*$ is a free graded $\mathbb C$-algebra with finitely many generators in each degree. We have a morphism of exact sequences of dg $A^0$-modules:
	\begin{equation}\label{WellComparison}\xymatrix{\drham{}{}{\deri{\wt{A}}{}{0}}\underset{\deri{\wt{A}}{}{0}}\otimes\dera{A}\ar[r]\ar[d] & 
	\drham{}{}{\dera{\wt{A}}}\underset{\dera{\wt{A}}}\otimes\dera{A}\ar[r]\ar[d] & \drham{}{\deri{\wt{A}}{}{0}}{\dera{\wt{A}}}\underset{\dera{\wt{A}}}\otimes\dera{A}\ar[d]\ar[r] & 0\ar[d]\\
	\drham{}{}{\deri{A}{}{0}}\underset{\deri{A}{}{0}}\otimes\dera{A}\ar[r] & \drham{}{}{\dera{A}}\ar[r] & \drham{}{\deri{A}{}{0}}{\dera{A}}\ar[r] & 0.}\end{equation}
Let $\mathfrak I^\bullet$ be the kernel of $\dera{\wt{A}}\twoheadrightarrow\dera{A}$, then the kernel of the left vertical arrow in (\ref{WellComparison}) is $\mathfrak I^0/(\mathfrak I^0)^2\underset{\deri{\wt{A}}{}{0}}\otimes\dera{A}$. Since $\deri{A}{}{0}$ is smooth over $\mathbb C$, $\mathfrak I^0/(\mathfrak I^0)^2$ is a projective $\deri{A}{}{0}$-module, and there is a section $\sigma$ of $\mathfrak I^1/\mathfrak I^0\mathfrak I^1\overset{\delta}\twoheadrightarrow\mathfrak I^0/(\mathfrak I^0)^2$. On the other hand, applying the K\"ahler differential, we have 
	\begin{equation*}\mathfrak I^1/\mathfrak I^0\mathfrak I^1\longrightarrow\drham{}{}{\dera{\wt{A}}}\underset{\dera{\wt{A}}}\otimes\dera{A},\quad
	\mathfrak I^1/\mathfrak I^0\mathfrak I^1\longrightarrow\drham{}{\deri{\wt{A}}{}{0}}{\dera{\wt{A}}}\underset{\dera{\wt{A}}}\otimes\dera{A}.\end{equation*} 
We compose these maps with $\sigma$ and obtain inclusions of the suspension of $\mathfrak I^0/(\mathfrak I^0)^2\underset{\deri{\wt{A}}{}{0}}\otimes\dera{A}$ into the middle and right non-trivial terms of the first row of (\ref{WellComparison}). Now we have $\mathfrak I^0/(\mathfrak I^0)^2\underset{\deri{\wt{A}}{}{0}}\otimes\dera{A}$ (or its suspension) sitting inside all three non-trivial terms of the first row of (\ref{WellComparison}). Dividing by the corresponding dg submodules, we still get an exact sequence (by construction), and the homotopy type of the middle term does not change (there we divide by an acyclic submodule). Since $A^*$ is freely generated over $A^0$ by projective modules, we notice that after this division the right term of the first row in (\ref{WellComparison}) becomes quasi-isomorphic to the right term of the second row. For the left term this is true by construction. Thus we conclude that the second from the left vertical arrow in (\ref{WellComparison}) is a quasi-isomorphism.\end{proof}%

\medskip

Proposition \ref{ComputingCotangent} justifies the following Definition, which is a slight generalization of the notion of a standard dg algebra in \cite{BBJ13}.

\begin{definition} We will say that a differential non-positively graded $\mathbb C$-algebra $\dera{A}$ is {\it well presented}, if it satisfies the conditions of Proposition \ref{ComputingCotangent} with the sequence $\{\promo^k\}_{k<0}$ being finite.\end{definition}
\begin{remark} The property of being well presented is not invariant with respect to weak equivalences of dg algebras. This notion is introduced to single out useful representatives of a given weak equivalence class. In particular, if $\dera{A}$ is well presented, $\drham{}{}{\dera{A}}$ is a perfect dg $\dera{A}$-module, with the underlying graded module freely generated by $\drham{}{}{\deri{A}{}{0}}$ and $\{\promo^k\}_{k<0}$.

On the other hand, the requirement that the sequence $\{\promo^k\}_{k<0}$ is finite implies that $\dera{A}$ is \emph{homotopically of finite presentation} (e.g.\@ \cite{To10} \S4.3). This property is invariant with respect to weak equivalences. From now on we will always assume that our dg algebras are well presented.\end{remark}

There is another property of dg $\mathbb C$-algebras that is not invariant with respect to weak equivalences, but is very helpful nonetheless. The authors have learned to appreciate its usefulness from Dominic Joyce.

\begin{definition}\label{DefMinimalAlgebra} Let $\dera{A}=(\undera{A},\homdi{})$ be a differential non-positively graded $\mathbb C$-algebra, and let $\pt\colon\dera{A}\rightarrow\mathbb C$ be \emph{a classical point}. We will say that \emph{$\dera{A}$ is minimal at $\pt$}, if $\homdi{}=0$ on $\drham{}{}{\dera{A}}|_\pt$, i.e.\@ modulo the kernel of $\pt$.\end{definition}

The following statement is known and obvious.

\begin{proposition} Let $\dera{A}$ be a well presented dg $\mathbb C$-algebra, and let $\pt\colon\dera{A}\rightarrow\mathbb C$. There are $f\in\deri{A}{}{0}$ and a surjective weak equivalence $\pi\colon\dera{A}[f^{-1}]\rightarrow\deri{A}{1}{\bullet}$ of well presented dg $\mathbb C$-algebras, s.t.\@ $\pt$ factors through $\pi$, and $\deri{A}{1}\bullet$ is minimal at $\pt$.\end{proposition}
\hide{\begin{proof} We can choose a finite sequence of elements of $\dera{A}$ in negative degrees, s.t.\@ modulo square of the kernel of $\pt$ these elements provide a basis for a complement to the kernel of $\homdi{}$. Let $\mathfrak A\subseteq\dera{A}$ be the dg ideal generated by these elements. As there are only finitely many of them, localizing around $\pt$, if necessary, we can assume that $\mathfrak A$ is acyclic. Then $\deri{A}{1}\bullet:=\dera{A}/\mathfrak A$.\end{proof}}

\medskip

In addition to differential non-positively graded $\mathbb C$-algebras we will need to work with \emph{differential non-positively graded $\cinfty$-rings} (e.g.\@ \cite{Dima}). By definition these are non-positively graded differential $\mathbb R$-algebras, that have, in addition, a $\cinfty$-structure on the degree $0$-component, extending the commutative $\mathbb R$-algebra structure. 

Notice that given two dg $\cinfty$-rings, that are Hausdorff with respect to the natural Fr\'echet topology, the set of dg $\cinfty$-morphisms between them equals the set of all morphisms between the underlying dg commutative algebras (e.g.\@ \cite{TopCha}). Correspondingly almost everything we do in this section extends verbatim to such dg $\cinfty$-rings, e.g.\@ a dg $\cinfty$-ring $\cring{A}{\bullet}$ is \emph{well presented}, if $\cring{A}0$ is the ring of $\cinfty$-functions on a smooth manifold, and the underlying graded $\cring{A}0$-algebra of $\cring{A}\bullet$ is freely generated by finitely many finitely generated projective $\cring{A}0$-modules. To distinguish dg commutative algebras from dg $\cinfty$-rings, we will write the latter in italics.

One difference between commutative algebras and $\cinfty$-rings, that we would like to mention, is the construction of K\"ahler differentials. Of course it is wrong to apply the usual definition to the case of $\cinfty$-rings. One has to take closures of the ideals and sub-modules with respect to Fr\'echet topology. This is what we will always do, without mentioning it again.

\subsection{Graded mixed algebras and integrable distributions}\label{SectionDistributions}

Let $\dera{A}$ be a differential non-positively graded $\mathbb C$-algebra. The discussion in this section applies equally well to dg $\cinfty$-rings, in which case one should replace $\mathbb C$ with $\mathbb R$ throughout. Since construction of K\"ahler differentials is functorial, we have a differential non-positively graded $\dera{A}$-module $\drham{}{}{\dera{A}}$. Taking (graded) anti-symmetric powers of $\drham{}{}{\dera{A}}$ and equipping it with the de Rham differential we obtain $\drham{\bullet}{}{\dera{A}}$, which is a graded mixed complex. Recall (e.g.\@ \cite{PTVV13} \S 1.1) that {\it a graded mixed complex over $\mathbb C$} is a triple $(\grami{\bullet}{\bullet},\homdi{},\cirdi)$, where $\grami{\bullet}{\bullet}$ is a $\mathbb Z_{\geq 0}\times\mathbb Z$-graded $\mathbb C$-space, and $\homdi{},\cirdi$ are $\mathbb C$-linear differentials 
	\begin{equation*}\homdi{}\colon\grami{\bullet}{\bullet}\longrightarrow\grami{\bullet+1}{\bullet+0},\quad
	\cirdi\colon\grami\bullet\bullet\longrightarrow\grami{\bullet-1}{\bullet+1},\end{equation*}
satisfying $\homdi{}^2=\cirdi^2=\homdi{}\circ\cirdi+\cirdi\circ\homdi{}=0$. Following loc.\@ cit.\@ we say that a component $\grami{m}{n}$ has {\it degree} $m$ and {\it weight} $n$. The de Rham complex
	\begin{equation*}\drham{\bullet}{}{\dera{A}}:=\underset{n\geq 0}\bigoplus\,\Sym{n}{\suspense{\drham{}{}{\dera{A}}}}{\dera{A}},\end{equation*}
provides an example. Here $\Sym{n}{-}{\dera{A}}$ stands for the $n$-th symmetric tensor power over $\dera{A}$, and $\suspense{(-)}$ is the suspension functor, i.e.\@ it lowers the cohomological degree by $1$.\footnote{We will use $\drham{}{}{\dera{A}}$ to denote the actual complex of K\"ahler differentials, and $\drham{1}{}{\dera{A}}$ for the shifted version.} Degrees of elements of ${\drham{\bullet}{}{\dera{A}}}$ are just their cohomological degrees, while the weight of an $n$-form is $n$. The cohomological differential $\homdi{}$ is obvious, and $\cirdi$ is the de Rham differential. 

\smallskip

It is important to notice that de Rham complexes have the structure of {\it graded mixed algebras} over $\mathbb C$ (\cite{TV11}, \cite{P14}), i.e.\@ the usual wedge product of differential forms makes $\drham{\bullet}{}{\dera{A}}$ into an associative, graded commutative $\mathbb C$-algebra with respect to both the degree and weight gradings, and the two differentials satisfy the (graded) Leibniz condition.\footnote{The wedge product of differential forms is of course $\dera{A}$-linear, not only $\mathbb C$-linear. This is due to the fact that $\dera{A}$ is the weight $0$ component of $\drham{\bullet}{}{\dera{A}}$.} This is not just any dg mixed algebra: as a dg $\dera{A}$-algebra it is freely generated over $\dera{A}$ by its weight $1$ component. This fact is important for defining integrable distributions.

\begin{definition}\label{DefFol} (E.g.\@ \cite{P14}) {\it An integrable distribution on $\spec{\dera{A}}$} is given by a pair $(\grami\bullet\bullet,\anchord)$, where $\grami\bullet\bullet$ is a graded mixed algebra over $\mathbb C$, and 
	\begin{equation*}\anchord\colon\drham{\bullet}{}{\dera{A}}\longrightarrow\grami\bullet\bullet\end{equation*} 
is a morphism of graded mixed algebras, s.t.\@\begin{enumerate}
\item $\anchord\colon\dera{A}=\drham{0}{}{\dera{A}}\overset\cong\longrightarrow\grami{\bullet}{0}$,
\item $\grami\bullet{1}$ is a perfect dg $\dera{A}$-module,
\item as a dg algebra (i.e.\@ forgetting $\cirdi$) $\grami\bullet\bullet\cong\underset{n\geq 0}\bigoplus\,\Sym{n}{\grami\bullet{1}}{\grami\bullet{0}}$.\end{enumerate}
A \emph{morphism} between integrable distributions is a morphism of graded mixed algebras under $\drham{\bullet}{}{\dera{A}}$.\end{definition}
\begin{remark} Our definition of integrable distributions has redundant parts and at the same time is not the most general one. The redundant part is $\anchord$ since the de Rham complex construction is a left adjoint to the forgetful functor from graded mixed algebras to dg algebras (e.g.\@ \cite{TV11} Prop.\@ 2.3). Hence having specified $\grami\bullet\bullet$ with $\grami\bullet{0}=\dera{A}$ we can reconstruct $\anchord$. However, we will need $\anchord$ to pull back integrable distributions over morphisms of dg algebras, so it is useful to fix $\anchord$ as part of the structure.

Our definition is not the most general one, because we could relax it by demanding that $\drham\bullet{}{\dera{A}}\rightarrow\grami\bullet\bullet$ is only weakly equivalent as a morphism of graded mixed algebras under $\drham\bullet{}{\dera{A}}$ to a morphism that we have described. However, for practical purposes one always chooses such representatives, and it makes sense to put this in the definition.\end{remark}
\begin{remark} On the $\dera{A}$-linear dual of $\anchord\colon{\drham{}{}{\dera{A}}}\rightarrow\suspensi{\grami\bullet{1}}{-1}$, that we will denote by $\anchor\colon\lalge\rightarrow\tanga{\dera{A}}$, we have the structure of a dg Lie--Rinehart algebra over $\dera{A}$. Recall that \emph{a Lie--Rinehart algebra over $\dera{A}$} is given by a perfect dg $\dera{A}$-module $\lalge$, an $\dera{A}$-linear map $\anchor\colon\lalge\rightarrow\tanga{\dera{A}}$ called \emph{the anchor}, and a $\mathbb C$-linear Lie bracket on $\lalge$ making $\anchor$ a morphism of dg Lie algebras. 

The category of dg Lie--Rinehart algebras over $\dera{A}$ is (anti-)equivalent to the category of integrable distributions according to Def.\@ \ref{DefFol} (e.g.\@ \cite{Vaintrob}). Here is an explicit description of this correspondence. Since inverting a function $f\in\deri{A}{}{0}$ is equivalent to adjoining a free variable $x$ and then dividing by the ideal generated by $xf-1$, we immediately see that to define Lie--Rinehart algebras over $\dera{A}$ it is enough to do so over a cover $\{\dera{A}[f_i^{-1}]\}$, and similarly for integrable distributions. Thus we can assume that $\grami\bullet{1}$ is generated as an $\undera{A}$-module by a finite sequence of trivializable $\deri{A}{}{0}$-modules of finite rank.

Omitting $\cirdi$ we have that $\grami\bullet\bullet$ is free as a commutative algebra over $\dera{A}$. Therefore every derivation of $\grami\bullet\bullet$ is a $\grami\bullet\bullet$-linear combination of derivations of $\dera{A}$ and elements of $\dumo{\grami\bullet{1}}:=\hom{\deri{A}{}{\bullet}}(\grami\bullet{1},\deri{A}{}{\bullet})$. So we can write 
	\begin{equation*}\cirdi=\cirdi_0+\cirdi_1,\quad\cirdi_0\in\grami\bullet{1}\underset{\deri{A}{}{0}}\otimes\tanga{\deri{A}{}{0}},\quad\cirdi_1\in\grami\bullet{2}\underset{\dera{A}}\otimes\dumo{\grami\bullet{1}}
	\quad\deg\cirdi_0=\deg\cirdi_1=-1.\end{equation*} 
Interpreting $\cirdi_0$ as $\anchor\colon\lalge=\suspensi{\dumo{\grami\bullet{1}}}{1}\rightarrow\tanga{\dera{A}}$ we obtain the anchor map. Viewing $\cirdi_0$ as defined on all of $\grami\bullet\bullet$, not just on $\grami\bullet{0}$, is possible by choosing a flat connection $\nabla$ on the generating $\deri{A}{}{0}$-submodules of $\grami\bullet{1}$. Degree $-1$ isomorphism $\tau\colon\dumo{\grami\bullet{1}}\rightarrow\lalge$ allows us to translate $\cirdi_1$ into $\cirdi'_1\colon\lalge\underset{\dera{A}}\otimes\lalge\rightarrow\lalge$, $\deg\cirdi'_1=0$.\footnote{One should not forget the Koszul sign rule: $\cirdi'_1(\tau(\lambda_1),\tau(\lambda_2))=(-1)^{\lambda_1}\tau(\cirdi_1(\lambda_1,\lambda_2))$.} Using the obvious $\pi\colon\lalge\underset{\mathbb C}\otimes\lalge\rightarrow\lalge\underset{\dera{A}}\otimes\lalge$ we define $[-,-]\colon\lalge\underset{\mathbb C}\otimes\lalge\rightarrow\lalge$ as follows
	\begin{equation}\label{LieBracket}(l_1,l_2)\longmapsto\nabla_{\anchor(l_1)}l_2-(-1)^{l_1 l_2}\nabla_{\anchor(l_2)}l_1+\cirdi'_1(\pi(l_1\otimes l_2)).\end{equation}
Direct computation shows that this indeed gives us a Lie algebroid (e.g.\@ \cite{Vaintrob}).  Choosing another flat connection adds an $\dera{A}$-linear term to the right hand side of (\ref{LieBracket}), that is then absorbed into $\cirdi'_1$, i.e.\@ this construction does not depend on the choice of a connection. \hide{%
Indeed, let $\{x^t\}$ be generators of $\deri{A}{}{0}$ and let $\{y^i\}$, $\{z^l\}$ be two possible choices for generators of $\grami\bullet{1}$ over $\dera{A}$. We will denote by $\{\partial_{y^i}\}$ the corresponding generators of $\dumo{\grami\bullet{1}}$. We will denote the copy of $\{x^t\}$ that makes up a coordinate system together with $\{z^l\}$ by $\{w^t\}$. We can write $\forall l$ $z^l=\underset{i}\sum\,\gamma^l_i y^i$, $\gamma_i^l\in\deri{A}{}{0}$. Then we have
	\begin{equation*}\partial_{y^i}=\underset{l}\sum\,\gamma^l_i\partial_{z^l},\quad\partial_{x^t}=\partial_{w^t}+\underset{i,l}\sum\,\partial_{x^t}(\gamma^l_i)y^i\partial_{z^l}.\end{equation*}
Let $\{\mu_l^i\}$ be the inverse of the matrix $\{\gamma^l_i\}$. Then we can write
	\begin{equation*}\cirdi=\underset{i,j,k}\sum\,c^k_{i,j}y^i y^j\partial_{y^k}+\underset{i,t}\sum\,a^t_iy^i\partial_{x^t}=\underset{l,m,n}\sum\,d^n_{l,m}z^l z^m\partial_{z^n}+\underset{l,t}\sum\,b^t_lz^l\partial_{w^t}=\end{equation*}
	\begin{equation*}=\underset{i,j,k,l,m,n}\sum\,d^n_{l,m}\gamma^l_iy^i\gamma^m_jy^j\mu^k_n\partial_{y^k}+
	\underset{i,l,t}\sum\,b^t_l\gamma^l_i y^i\Big(\partial_{x^t}-\underset{j,k,m}\sum\,\partial_{x^t}(\gamma^m_j)y^j\mu^k_m\partial_{y^k}\Big).\end{equation*}
Comparing the first expression with the last we obtain
	\begin{equation*}c^k_{i,j}=\underset{l,m,n}\sum d^n_{l,m}\gamma^l_i\gamma^m_j\mu^k_n-
	\underset{l,m,t}\sum\,\Big(b^t_l\gamma^l_i\partial_{x^t}(\gamma^m_j)\mu^k_m+(-1)^{y^i y^j}b^t_l\gamma^l_j\partial_{x^t}(\gamma^m_i)\mu^k_m\Big).\end{equation*}
The definition of the bracket on $\lalge$ gives us
	\begin{equation*}[\tau(\partial_{y^i}),\tau(\partial_{y^j})]=\underset{k}\sum\,(-1)^{y^i}c^k_{i,j}\partial_{y^k}=\underset{k,l}\sum\,c^k_{i,j}\gamma^l_k\partial_{z^l}.\end{equation*}
On the other hand
	\begin{equation*}[\tau(\partial_{y^i}),\tau(\partial_{y^j})]=[\underset{l}\sum\,\gamma^l_i\tau(\partial_{z^l}),\underset{m}\sum\,\gamma_j^m\tau(\partial_{z^m})]=
	\underset{l,m,n}\sum(-1)^{z^l}\gamma^l_i\gamma^m_j d^n_{l,m}\tau(\partial_{z^n})+\end{equation*}
	\begin{equation}\label{End}+\underset{l,m}\sum\,\gamma^l_i\tau(\partial_{z^l})(\gamma^m_j)\tau(\partial_{z^m})-\underset{l,m}\sum\,(-1)^{z^l z^m}\gamma^m_j\tau(\partial_{z^m})(\gamma^l_i)\tau(\partial_{z^l}).\end{equation}
Now we compare this to the expression we found for $c^k_{i,j}$. We notice that $\partial_{z^n}=\underset{k}\sum\,\mu^k_n\partial_{y^k}$ and hence the first summand agrees with the first one in the expression for $c^k_{i,j}$. The last two sums in (\ref{End}) are not $0$, only if $\deg z_l,\deg z_m=-1$, in which case the last sum in the expression for $c^k_{i,j}$ agrees with the two sums in (\ref{End}) up to the sign. The sign is then brought by the coefficient before $c^k_{i,j}$ $(-1)^{y^i}=-1$. }%
Conversely, if we start with a Lie--Rinehart algebra, a choice of a flat connection on $\lalge$ gives us another bracket on $\lalge$, and the difference between the two brackets produces $\cirdi'_1$ as in (\ref{LieBracket}). The anchor gives $\cirdi_0$.\end{remark}

Since differential graded algebras correspond to (parts of) formal neighborhoods of the classical loci, it is very easy to construct integrable distributions in the derived setting, where integrability holds for degree reasons. The following example will be used repeatedly in this work.

\begin{example}\label{SimpleEx} For any well presented $\dera{A}$ with $\{\bundle{P}{k}{}\}_{k<0}$ being the generating bundles in negative degrees, the tangent complex 
	\begin{equation*}\tanga{\dera{A}}=\hom{\dera{A}}(\drham{}{}{\dera{A}},\dera{A})\end{equation*}
is a perfect complex, with the underlying graded module being freely generated over $\undera{A}$ by $\tanga{\deri{A}{}{0}}$ and 
	\begin{equation*}\bundle{E}{-k}{}:=\hom{A^0}(\bundle{P}{k}{},A^0),\quad k<0.\end{equation*} 
If we choose $m>0$, and a sub-bundle $\bundle{E}{m}{-}\subseteq\bundle{E}{m}{}$, the dg submodule $(\tanga{\dera{A}})_-\subseteq\tanga{\dera{A}}$ generated by $\bundle{E}{m}{-}\oplus(\underset{k>m}\bigoplus\bundle{E}{k}{})$ is clearly closed with respect to the Lie bracket (for degree reasons). Let $\bundle{P}{-m}{-}\subseteq\bundle{P}{-m}{}$ be the orthogonal complement of $\bundle{E}{m}{-}$, and let $\drham\bullet{}{\dera{A}}_-\subseteq\drham\bullet{}{\dera{A}}$ be the graded mixed ideal generated by $\suspensi{\big(\bundle{P}{-m}{-}\oplus(\underset{k<m}\bigoplus\,\bundle{P}{-k}{})\big)}{1}$. The same degree reasons that caused $(\tanga{\dera{A}})_-\subseteq\tanga{\dera{A}}$ to be closed with respect to the Lie bracket, imply that $\drham\bullet{}{\dera{A}}_-$ is just the usual dg ideal generated by $\suspensi{\big(\bundle{P}{-m}{-}\oplus(\underset{k<m}\bigoplus\,\bundle{P}{-k}{})\big)}{1}$. Defining 
	\begin{equation*}\drham\bullet{}{\dera{A}}_+:=\drham\bullet{}{\dera{A}}/
	\drham\bullet{}{\dera{A}}_-\end{equation*}
we obviously have an integrable distribution in the sense of Definition \ref{DefFol}.\end{example}

Quotients of derived schemes and stacks by integrable distributions are formulated in the category of formal derived stacks (e.g.\@ \cite{P14}). We will not review the theory of formal derived stacks here (e.g.\@ \cite{CPTVV}), but we will recall the definition of the quotient.

\begin{definition} (E.g.\@ \cite{P14}) Let $\dersch$ be a derived scheme and let $\grami\bullet\bullet$ be an integrable distribution on $\dersch$. \emph{A quotient of $\dersch$ by $\grami\bullet\bullet$} is a morphism of formal derived stacks $\dersch\rightarrow\quotient{\dersch}{\grami\bullet\bullet}$ that solves the following universal problem: given any morphism of formal derived stacks $\phi\colon\dersch\rightarrow\dersch'$ and any morphism of integrable distributions $\drham\bullet{}{\dersch/\dersch'}\rightarrow\grami\bullet\bullet$ there is a unique $\dersch/\grami\bullet\bullet\rightarrow\dersch'$ that factorizes $\phi$.\end{definition}

\begin{example} Consider $\grami\bullet\bullet=\drham\bullet{}{\dersch}$. In this case the universal property forces $\dersch/\grami\bullet\bullet$ to have trivial cotangent complex everywhere, i.e.\@ the value of $\dersch/\grami\bullet\bullet$ on $\spec{\dera{A}}$ is the value of $\dersch$ on $\spec{\reduced{\dera{A}}}$. This $\dersch/\grami\bullet\bullet$ has a name: \emph{the de Rham space of} $\dersch$ (e.g.\@ \cite{GR14}). We will denote it by $\derhamspace{\dersch}$.\end{example}
The example of de Rham space of a scheme is a particular case of a relative construction when we consider $\dersch$ over the point. The general relative case is one of our main interests in this paper, as we would like to divide by distributions that are the tangent distributions of some morphism. The following simple statement gives a way to compute the result of such division.

\begin{proposition}(E.g.\@ \cite{P14}) Let $\dersch\rightarrow\dersch'$ be a morphism of derived schemes, and let $\grami\bullet\bullet$ be the relative de Rham complex. Then $\dersch/\grami\bullet\bullet$ is a homotopy limit of the following diagram
	\begin{equation}\label{QuotientPullBack}\begin{tikzcd} & \derhamspace{\dersch}\ar[d]\\
	\dersch'\ar[r] & \derhamspace{\dersch'}.\end{tikzcd}\end{equation}
\end{proposition}
\begin{proof} The argument is exactly the same as in the case of the absolute de Rham space, but applied in the category of derived schemes over $\dersch'$.\end{proof}%
The last proposition has a corollary, that is very important for this paper.
\begin{corollary} Let $\dersch\rightarrow\dersch'$ be a morphism of derived schemes, and let $\grami\bullet\bullet$ be the relative de Rham complex. Suppose that $\reduced{\dersch}\overset\cong\longrightarrow\reduced{\dersch'}$, then $\dersch'\cong\dersch/\grami\bullet\bullet$.\end{corollary}
\begin{proof} The assumption that $\reduced{\dersch}\cong\reduced{\dersch'}$ implies that $\derhamspace{\dersch}\cong\derhamspace{\dersch'}$, and then obviously $\dersch'$ is a homotopy pullback of (\ref{QuotientPullBack}).\end{proof}%

In the notation of the previous corollary, suppose we would like to reconstruct $\dersch$ from $\dersch'$. It was explained in \cite{P14}, that if there is a shifted symplectic structure on $\dersch$ and $\grami\bullet\bullet$ is a Lagrangian distribution, we can reconstruct $\dersch$ from $\dersch'$ as a derived critical locus of a shifted potential defined on $\dersch'$. 

It is the main purpose of this paper to show that in the case of moduli spaces of sheaves on Calabi--Yau manifolds, we can always choose a Lagrangian distribution, which, after division, allows us to encode the entire moduli space as a shifted potential on a moduli space of $Spin(7)$-instantons. However, since this requires working with stacks, rather than just derived schemes, we postpone the actual construction of the potential to another paper.

\subsection{Pull-backs and derived foliations}

Recall (\cite{TV11} \S2) that the category of mixed complexes admits a model structure with weak equivalences being quasi-isomorphisms with respect to $\homdi{}$ and fibrations being surjective maps (in negative degrees). The same model structure (and without the restriction on complexes to be non-positively graded) is obtained by applying the general machinery of algebras over operads (e.g.\@ \cite{Hinich}), where we view the mixed algebras as commutative algebras together with a unary operation of degree $-1$. Similar constructions are applicable also in presence of the additional grading by weight. In particular every graded mixed algebra is fibrant, which justifies the following definition.

\begin{definition}\label{DefEquivalenceDist} Two integrable distributions $\drham{\bullet}{}{\deri{A}{}\bullet}\rightarrow\gramii\bullet\bullet{1}$, $\drham{\bullet}{}{\deri{A}{}\bullet}\rightarrow\gramii\bullet\bullet{2}$ are \emph{equivalent}, if there is a third distribution $\drham{\bullet}{}{\deri{A}{}\bullet}\rightarrow\gramii\bullet\bullet{3}$ and weak equivalences $(\gramii\bullet\bullet{3},\homdi{})\overset\simeq\longrightarrow(\gramii\bullet\bullet{1},\homdi{})$, $(\gramii\bullet\bullet{3},\homdi{})\overset\simeq\longrightarrow(\gramii\bullet\bullet{2},\homdi{})$ making the following diagram of graded mixed algebras commutative
	\begin{equation*}\begin{tikzcd} \gramii\bullet\bullet{1} & \drham{\bullet}{}{\deri{A}{}\bullet} \ar[r]\ar[l] \ar[d] & \gramii\bullet\bullet{2}  \\
	& \gramii\bullet\bullet{3}.\ar[lu,"\simeq"]\ar[ru,"\simeq",swap] &\end{tikzcd}\end{equation*}
\end{definition}
Given a morphism $\phi\colon\deri{A}{1}\bullet\rightarrow\deri{A}{2}\bullet$ we have the natural \emph{pull-back functor} from the category of integrable distributions over $\deri{A}{1}\bullet$ to the category of integrable distributions over $\deri{A}{2}\bullet$. Explicitly  $\drham{\bullet}{}{\deri{A}{1}\bullet}\rightarrow\gramii\bullet\bullet{1}$ is mapped to the homotopy push-out of
	\begin{equation}\label{DistributionPushout}\begin{tikzcd} \drham{\bullet}{}{\deri{A}{1}\bullet} \ar[r] \ar[d] & \gramii\bullet\bullet{1}  \\
	\drham\bullet{}{\deri{A}{2}\bullet} &\end{tikzcd}\end{equation}
computed in the category of graded mixed algebras. We will denote the result by $\pullback{\phi}{\gramii\bullet\bullet{1}}$. To convince ourselves that this is still an integrable distribution we need to take a closer look at the actual computation of this homotopy push-out. We could use the usual categorical push-out, if $\drham\bullet{}{\deri{A}{1}{}}$ was cofibrant and either one of $\drham{\bullet}{}{\deri{A}{1}\bullet}\rightarrow\drham{\bullet}{}{\deri{A}{2}\bullet}$ or $\drham{\bullet}{}{\deri{A}{1}\bullet}\rightarrow\gramii\bullet\bullet{1}$ was a cofibration. Naturally we would like to have the latter as a cofibration.

It is helpful here to make use of the category of Lie--Rinehart algebras, which is (anti-)equivalent to the category of integrable distributions, but not to the category of all graded mixed algebras under $\drham{\bullet}{}{\deri{A}{1}\bullet}$. The latter is much bigger. Being the category of algebras over a (colored) operad, also the category of Lie--Rinehart algebras possesses a model structure with weak equivalences being quasi-isomorphisms and fibrations being surjective maps. 

Viewing the anchor as a morphism of Lie--Rinehart algebras and factorizing it, we obtain a resolution with a surjective anchor.  In terms of integrable distributions a surjective anchor corresponds to a cofibration of graded mixed algebras. Thus taking resolutions, if necessary, we can always assume that $\drham{\bullet}{}{\deri{A}{1}\bullet}\rightarrow\gramii\bullet\bullet{1}$ is a cofibration. This shows that the homotopy push-out in (\ref{DistributionPushout}) produces another integrable distribution.

It remains to clarify when $\drham\bullet{}{\deri{A}{1}{}}$ is cofibrant. Since the homotopy push-out is naturally mapped to the usual push-out of (\ref{DistributionPushout}), we can use the usual push-out as long as $\drham\bullet{}{\deri{A}{1}{}}$ is \emph{locally} cofibrant on $\spec{\deri{A}{1}\bullet}$. The de Rham complex construction is a left Quillen functor from the category of dg algebras to that of graded mixed algebras (e.g.\@ \cite{TV11} Prop.\@ 2.3). Hence the usual push-out in (\ref{DistributionPushout}) is also a homotopy push-out, if $\deri{A}{1}\bullet$ is locally cofibrant, which is always true for well presented dg $\cinfty$-rings.

\medskip

Having the ability to substitute any integrable distribution with a cofibrant resolution, that is still an integrable distribution, we can compute pullbacks. For example given $f\in\deri{A}{}{0}$ and a cofibrant $\drham\bullet{}{\dera{A}}\rightarrow\grami\bullet\bullet$ we can restrict this distribution to $\dera{A}[f^{-1}]$, i.e.\@ we can pull it back along the inclusion. We will denote the result by $\grami\bullet\bullet[f^{-1}]$ and in this way we obtain a pre-sheaf of cofibrant integrable distributions on $\spec{\dera{A}}$.

Now we are ready to formalize an important property of Example \ref{SimpleEx}. There the morphism $\drham\bullet{}{\dera{A}}\rightarrow\grami\bullet\bullet$ is given as a factorization by a graded mixed ideal of $\drham\bullet{}{\dera{A}}$. We would like to be able to single out integrable distributions that have such presentation. Moreover, we would like this property to hold locally after pull-back to any dg algebra that is weakly equivalent to $\dera{A}$. This is a rather strong condition.

\begin{definition}\label{DefDerivedFoliation} Let $\dera{A}$ be a well presented dg algebra. An integrable distribution $\anchord\colon\drham{\bullet}{}{\dera{A}}\rightarrow\grami\bullet\bullet$ is a \emph{derived foliation}, if $\forall\pt\colon\dera{A}\rightarrow\mathbb C$ there is $f\in\deri{A}{}{0}$ and a weak equivalence $\phi\colon\dera{A}[f^{-1}]\overset\simeq\longrightarrow\deri{A}{1}{\bullet}$, s.t.\@ $\deri{A}{1}{\bullet}$ is well presented and minimal at $\pt$ and $\pullback{\phi}{\grami\bullet\bullet}\simeq\gramii\bullet\bullet{1}$ under $\drham\bullet{}{\deri{A}{1}\bullet}$, where $\gramii\bullet\bullet{1}$ is a quotient of $\drham\bullet{}{\deri{A}{1}\bullet}$ by a graded mixed ideal with generators having weight $1$.\end{definition}
It is clear from this definition that the property of being a derived foliation is invariant with respect to weak equivalences of integrable distributions and stable under pull-backs of distributions from one dg algebra to another. 

\begin{remark} By far not every integrable distribution is a derived foliation. In the notation of Def.\@ \ref{DefDerivedFoliation}, since $\deri{A}{1}\bullet$ is minimal at $\pt$ we have that $\homdi{}=0$ on $\drham\bullet{}{\deri{A}{1}\bullet}$ modulo the maximal ideal of $\pt$. Then the same holds for $\gramii\bullet\bullet{1}$, and we conclude that $\homology{*}(\drham\bullet{}{\deri{A}{1}\bullet})|_\pt\rightarrow\homology{*}(\gramii\bullet\bullet{1})|_\pt$ is surjective. This is a non-trivial cohomological condition on a distribution.\end{remark}
\begin{remark} Given a cofibration $\deri{A}{1}\bullet\rightarrow\deri{A}{2}\bullet$ and a surjective $\drham\bullet{}{\deri{A}{1}\bullet}\rightarrow\grami\bullet\bullet$, the usual pull-back of $\grami\bullet\bullet$ to $\deri{A}{2}\bullet$ is again surjective, and  it is also the homotopy pull-back. Thus, if every minimal dg algebra was cofibrant, any derived foliation could have been represented locally as a quotient by a graded mixed ideal over \emph{every} choice of a dg algebra, not necessarily minimal. This observation is helpful in the $\cinfty$-case, since every well presented dg $\cinfty$-ring is locally cofibrant.\end{remark}

\medskip

Having defined integrable distributions on derived affine schemes we would like to describe such distributions on arbitrary derived schemes. For practical purposes it is useful to do this in terms of atlases. Then a distribution would be defined on each chart and coherent gluing data would be provided on intersections. 

To keep away from unnecessary generality we assume that \emph{our derived schemes are separated}, in particular intersections of derived affine charts are themselves derived affine schemes. We will denote by $\unspa$ the topological space of classical points, underlying a derived scheme $\dersch$. The topology on $\unspa$ will depend on the type of functions we work with: Zariski in the case of dg algebras and analytic in the case of dg $\cinfty$-rings.

\smallskip

Let $\big\{\spec{\deri{A}{j}{\bullet}}\big\}_{j\in J}$ be a derived affine atlas on $\dersch$. Again for the sake of practicality we assume that the corresponding open cover of $\unspa$ is locally finite. Moreover, we will always assume that any open cover of $\unspa$ has a locally finite refinement. For every ordered subset $\{j_0,\ldots,j_k\}\subseteq J$ of not necessarily distinct elements we have a derived affine scheme
	\begin{equation*}\spec{\deri{A}{j_0,\ldots,j_k}\bullet}:=\underset{0\leq s\leq k}\bigcap\spec{\deri{A}{j_s}\bullet}.\footnote{We write $\bigcap$ to mean the homotopy product over $\dersch$. Correspondingly we will write the pullback, e.g.\@ of differential forms, from a factor to this product as a restriction.}\end{equation*}

\begin{definition}\label{DefAtlasDistribution} \emph{An integrable distribution on an atlas $\big\{\spec{\deri{A}{j}{\bullet}}\big\}_{j\in J}$} is given by the following data
	\begin{equation*}\anchord=\big\{\anchord_{j_0,\ldots,j_k}\colon\rahm{\deri{A}{j_0,\ldots,j_k}{\bullet}}{\bullet}\rightarrow\gramii{\bullet}{\bullet}{j_0,\ldots,j_k}\;|\;k\geq 0,\,j_0,\ldots,j_k\in J\big\},\end{equation*}
	\begin{equation*}\conne{}=\big\{\conne{\makebo{j}{k}{s}}\colon(\gramii{\bullet}{\bullet}{j_0,\ldots,\widehat{j_s},\ldots,j_k})|_{\spec{\deri{A}{j_0,\ldots,j_k}{\bullet}}}\overset\simeq\longrightarrow\gramii{\bullet}{\bullet}{j_0,\ldots,j_k}
	\,|\,0\leq s\leq k\big\}\end{equation*}
subject to conditions: $\forall k\geq 0$ $\forall\{j_0,\ldots,j_k\}\subseteq J$ $(\gramii{\bullet}{\bullet}{j_0,\ldots,j_k},\anchord_{j_0,\ldots,j_k})$ is an integrable distribution on $\spec{\deri{A}{j_0,\ldots,j_k}\bullet}$ and, suppressing the notation for the pull-back of differential forms and bundles from $\spec{\deri{A}{j_0,\ldots,\widehat{j_s},\ldots,j_k}\bullet}$ to $\spec{\deri{A}{j_0,\ldots,j_k}\bullet}$, we have that $\forall s$ $\rahm{\deri{A}{j_0,\ldots,\widehat{j_s},\ldots,j_k}{\bullet}}{\bullet}\longrightarrow\gramii{\bullet}{\bullet}{j_0,\ldots,j_k}$ factors through $\conne{\makebo{j}{k}{s}}$, and the obvious cosimplicial identities hold for $\big\{\conne{\makebo{j}{k}{s}}\big\}$.\end{definition}
This is not the most general way to define an integrable distribution on an atlas. For example one could require that cosimplicial identities hold for the comparison quasi-isomorphisms $\big\{\conne{\makebo{j}{k}{s}}\big\}$ only up to coherent homotopy. However, this level of generality will be sufficient for our purposes.

\begin{remark} Notice that intersections of derived affine charts are defined up to weak equivalences. This implies that integrable distributions (and later shifted symplectic structures and isotropic structures) are defined on all possible choices of intersections of charts.\end{remark}

\section{Symplectic and Lagrangian structures}\label{SectionSymplecticStructures}

Having taken care of cotangent complexes and integrable distributions we turn to the main objects of our study: shifted symplectic structures and Lagrangian distributions. First we recall from \cite{PTVV13} the basic notions of homotopically closed and symplectic forms, and then give an explicit description of such structures on atlases. For practical purposes the main outcome is the construction on spaces of classical points of sheaves of symplectic forms and corresponding sheaves of isotropic distributions. Finding global sections of the latter is the main goal of this paper.

\subsection{Homotopically closed and symplectic differential forms}

Given a well-presented differential non-positively graded $\mathbb C$-algebra $\dera{A}$ we have the graded mixed complex of differential forms $(\drham{\bullet}{}{\dera{A}},\dedi,\homdi{})$. Here $\dedi$ is the de Rham differential and $\homdi{}$ is the cohomological differential. Presence of $\homdi{}$ allows us to formulate a homotopical version of the notion of a de Rham closed form. 

This was done in \cite{PTVV13} by using negative cyclic homology and explicitly in terms of infinite series of forms, that are cocycles for the sum of $\dedi$ and $\homdi{}$. We will call such structures \emph{homotopically closed differential forms}. 

\smallskip

Here is a quick summary of this construction. Recall that $\drham{\bullet}{}{\dera{A}}$ has a double grading by weight and cohomological degree, where $\dedi$, $\homdi{}$ have weights $1$, $0$ and degrees $-1$, $1$ respectively. Consider a parameter $\para$ to which we assign weight $-1$ and degree $2$. Then, following \cite{PTVV13} \S1.1, we define
	\begin{equation*}\forall k\geq 0,\forall n\leq 0\quad\negacy{n}{k}\big(\drham\bullet{}{\dera{A}}\big):=\Big\{\underset{j=0}{\overset{\infty}\sum}t^j\nform{j}\Big\},\end{equation*}
where each $\nform{j}$ is a $k+j$-form of degree $n-2j$, i.e.\@ chosen so that $t^j\nform{j}$ has degree $k$ and weight $n$. It is clear that $\homdi{}+t\dedi$ maps $\negacy{n}{k}\big(\drham\bullet{}{\dera{A}}\big)$ to $\negacy{n+1}{k}\big(\drham\bullet{}{\dera{A}}\big)$ and we can define a dg $\mathbb C$-space
	\begin{equation*}\negacy{\bullet}{k}\big(\drham\bullet{}{\dera{A}}\big):=\Big(\underset{n\leq 0}\bigoplus\,\negacy{n}{k}\big(\drham\bullet{}{\dera{A}}\big),\homdi{}+t\dedi\Big).\end{equation*}
This construction is called \emph{negative cyclic complex of weight $k$} in \cite{PTVV13}, and can be applied to any graded mixed complex (without the condition on cohomological degrees to be non-positive). It is important to understand functoriality of this construction, in particular its interaction with quasi-isomorphisms {with respect to $\homdi{}$}.

\begin{proposition} (\cite{PTVV13} Prop.\@ 1.3) The construction of weighted negative cyclic complexes is a right Quillen functor from the category of graded mixed complexes and quasi-isomorphisms (with respect to $\homdi{}$) to the category of dg complexes and quasi-isomorphisms. In particular this functor preserves quasi-isomorphisms.\end{proposition}
\begin{definition}(\cite{PTVV13} \S1.2) Let $\dera{A}$ be a well presented dg $\mathbb C$-algebra, $k\geq 0$, $n\leq 0$. \emph{A homotopically closed $k$-form of degree $n$} on $\spec{\dera{A}}$ is an element $\symstra\in\negacy{n-k}{k}\big(\drham\bullet{}{\dera{A}}\big)$ that is a cocycle with respect to $\homdi{}+t\dedi$.\footnote{Notice the shift by $k$ in $\negacy{n-k}{k}$. This is done so that a $2$-form of degree $n$ would define a morphism of degree $n$ from the tangent to the cotangent complexes.} Two such forms are \emph{equivalent}, if their difference is a $(\homdi{}+t\dedi)$-coboundary.\end{definition}
Since $\symstra$ is a series, it has a free term -- $\nform{0}$. This term will play a special role, and we will denote it by $\symfo$. Notice that $\symfo$ is not necessarily de Rham closed, but it has to be a $\homdi{}$-cocycle. In particular, when $k=2$, it defines a (shifted) morphism of dg $\dera{A}$-modules $\tanga{\dera{A}}\rightarrow\drham{}{}{\dera{A}}$.

\begin{definition}\label{DefSymAffine} (\cite{PTVV13} \S1.2) Let $\dera{A}$ be a well presented dg $\mathbb C$-algebra, $n\leq 0$. \emph{An $n$-shifted symplectic form} on $\spec{\dera{A}}$ is a homotopically closed $2$-form of degree $n$ $\symstra$, s.t.\@ $\symfo$ defines a quasi-isomorphism $\tanga{{\dera{A}}}\overset\simeq\longrightarrow\suspensi{\drham{}{}{\dera{A}}}{n}$.\end{definition}
As shifted symplectic forms are infinite series, they can be rather complicated to deal with. The following example provides a strict version, that is very useful in practice. As always we assume that $\dera{A}$ is well presented.
\begin{example} \label{StriFoEx} We will say that a $-2$-shifted symplectic form $\symstra$ is \emph{strict},\footnote{In \cite{BBJ13} such symplectic structure is said to be in `strong Darboux form'.}  if $\symstra=\symfo\in\drham{2}{}{\dera{A}}$, and, moreover, locally on $\spec{\deri{A}{}{0}}$ we have
	\begin{equation*}\symfo=\underset{1\leq i\leq m}\sum\dedi x_i\wedge\dedi z_i+\underset{1\leq j\leq n}\sum\dedi y_j\wedge\dedi y_j,\end{equation*}
where $x_i$, $y_j$, $z_i$ lie in $\dera{A}$, have cohomological degrees $0$, $-1$, $-2$ respectively, and $\symfo$ defines \underline{an isomorphism}  
	\begin{equation*}\tanga{\dera{A}}\overset{\cong}\longrightarrow\suspensi{\drham{}{}{\dera{A}}}{-2}.\end{equation*}
Having a $-2$-shifted strict symplectic form $\symfo$ on $\spec{\dera{A}}$ puts some restrictions and structures on $\dera{A}$:\begin{enumerate}
	\item the projective $A^0$-modules $\{\promo^k\}$ that generate $\undera{A}$ over $A^0$ can be non-trivial only in degrees $-1$ and $-2$, 
	\item $\symfo_2$ defines an $A^0$-linear isomorphism between $\drham{}{}{A^0}$ and $\dumo{\promo^{-2}}=:\ubun$, 
	\item $\symfo_2$ defines a non-degenerate symmetric $A^0$-bilinear form on $\dumo{\promo^{-1}}=:\tosplit{E}$.\end{enumerate}	
\end{example}
The following proposition is part of a local strictification result, valid for all negative shifts, that was proved in \cite{BBJ13}.

\begin{proposition}\label{LocalD} (\cite{BBJ13} Thm.\@ 5.18) Let $\dersch$ be a derived scheme, and let $\symstra$ be a $-2$-shifted symplectic structure on $\dersch$. Any $\pt\colon\spec{\mathbb C}\rightarrow\dersch$ factors through an \'etale $\phi\colon\spec{\dera{A}}\rightarrow\dersch$ with $\dera{A}$ being well presented, s.t.\@ $\phi^*(\symstra)$ is equivalent to a strict $-2$-shifted symplectic form on $\spec{\dera{A}}$.\end{proposition}
For a symplectic form $\symstra$ to be defined over all of $\dersch$, which is not necessarily an affine derived scheme, means that we have a chosen symplectic form on each derived affine chart, and these structures coherently glue on intersections. We would like to give an explicit description of all this in terms of an atlas. In the notation of Section \ref{SectionDistributions} we have the following

\begin{definition}\label{DefSymAt} \emph{An $n$-shifted symplectic form on an atlas $\big\{\spec{\deri{A}{j}{\bullet}}\big\}_{j\in J}$} is given by the following data
	\begin{equation*}\symstra=\big\{\symstra_{j_0,\ldots,j_k}\in\negacy{n-2-k}{2}\big(\derham{\deri{A}{j_0,\ldots,j_k}\bullet}{\mathbb C}\bullet\big)\;|\;k\geq 0,\,j_0,\ldots,j_k\in J\big\},\end{equation*}
subject to conditions: $\forall j\in J$ $\symstra_{j}$ is an $n$-shifted symplectic form on $\spec{\deri{A}{j}\bullet}$ and, suppressing the notation for the pull-back of differential forms from $\spec{\deri{A}{j_0,\ldots,\widehat{j_s},\ldots,j_k}\bullet}$ to $\spec{\deri{A}{j_0,\ldots,j_k}\bullet}$, we have $\forall k\geq 1,\forall\{j_0,\ldots,j_k\}\subseteq J$
	\begin{equation*}(\homdi{}+t\dedi)(\symstra_{j_0,\ldots,j_k})=\underset{0\leq s\leq k}\sum(-1)^{s}\symstra_{j_0,\ldots,\widehat{j_s},\ldots,j_k}.\end{equation*}
\end{definition}

\subsection{Isotropic and Lagrangian distributions}

Having described a homotopical version of symplectic forms, we turn to a homotopical version of isotropic distributions.

\begin{definition}\label{DefIsoDist} (\cite{P14}) Let $\dera{A}$ be a well presented dg $\mathbb C$-algebra, and let $\symstra$ be an $n$-shifted symplectic structure on $\spec{\dera{A}}$. {\it An isotropic distribution} 
	\begin{equation*}\anchord\colon(\drham\bullet{}{\dera{A}},\symstra)\longrightarrow(\grami\bullet\bullet,\lagras)\end{equation*}
consists of an integrable distribution $\anchord\colon\drham\bullet{}{\dera{A}}\rightarrow\grami\bullet\bullet$ and an an element $\lagras\in\negacy{n-3}{2}(\grami\bullet\bullet)$ s.t.\@ $\anchord(\symstra)=(\homdi{}+t\cirdi)(\lagras)$. {\it Two isotropic structures $\lagras_1$, $\lagras_2$ on $\anchord$ are equivalent}, if $\lagras_1-\lagras_2$ is a $(\homdi{}+t\cirdi)$-coboundary. \end{definition} 
\begin{example}\label{ExIso} In the setting of Example \ref{StriFoEx}, suppose we choose an integrable distribution as in Example \ref{SimpleEx}, i.e.\@ it is given by a sub-module $\nesplit{E}\subseteq\tosplit{E}$. If we make sure that $\symfo|_{\nesplit{E}}=0$, we obtain an isotropic structure given by $\lagras=0$.\end{example}
As in the case of symplectic structures $\lagras$ is a series $\underset{j\geq 0}\sum t^j\lagfo{j}$, where $\lagfo{j}$ is a $2+j$-form of cohomological degree $n-3-j$. The first coefficient $\lagfo{0}$ is especially important, and we will denote it by $\lagra$. Switching to the dual $\anchor\colon\lalge\longrightarrow\tanga{\dera{A}}$ of $\anchord$ we observe that, wherever $\anchor^*(\symfo)=0$, we necessarily have $\homdi{}(\lagra)=0$. This holds, in particular, on the (homotopy) kernel $\homke{}$ of $\anchor$. I.e.\@ we obtain a morphism of dg $\dera{A}$-modules
	\begin{equation}\label{IsoMap}\lagra\colon\homke{}\longrightarrow\dumo{\lalge}[n-1]\cong\grami\bullet{1}[n-2].\end{equation}
\begin{definition}\label{DefLagrangian} (\cite{P14}) An isotropic distribution $\anchord\colon(\drham\bullet{}{\dera{A}},\symstra)\rightarrow(\grami\bullet\bullet,\lagras)$ is {\it Lagrangian}, if (\ref{IsoMap}) is a quasi-isomorphism.\end{definition}
To compute a homotopy kernel of the anchor one can choose a surjective representative $\wt{\lalge}\rightarrow\tanga{\dera{A}}$ of $\lalge\rightarrow\tanga{\dera{A}}$ and then take the usual kernel. Let's do this in an example.

\begin{example} Continuing Example \ref{ExIso}, we claim that we obtain a Lagrangian distribution, if the rank of $\nesplit{E}$ is half the rank of $\tosplit{E}$, i.e.\@ if this is a maximally isotropic sub-bundle in the usual sense. Let's compute the homotopy kernel of the anchor map in this case to check that the two definitions agree.

So suppose that $\nesplit{E}$ is maximally isotropic in the usual sense, and let $\posplit{E}$ be a maximally isotropic complement to $\nesplit{E}$. We obtain a surjective replacement of the anchor map, if we define $\widetilde{\lalge}$ to be the dg $\dera{A}$-module generated by 
	\begin{equation*}\tanga{\deri{A}{}{0}},\quad\tanga{\deri{A}{}{0}}\oplus\posplit{E}\oplus\nesplit{E},\quad\posplit{E}\oplus\ubun\end{equation*} 
in degrees $0$, $1$ and $2$ respectively. The differentials on the added summands are the identities plus the differential on $\tanga{{\dera{A}}}$. Then it is easy to compute that the kernel of $\wt{\lalge}\rightarrow\tanga{\dera{A}}$ is generated over $\dera{A}$ by the complex $\tanga{\deri{A}{}{0}}\rightarrow\tosplit{E}/\nesplit{E}$ sitting in degrees $1$ and $2$. On the other hand, the pullback of ${\symfo}$ to $\wt{\lalge}$ is not $0$ anymore. There is an obvious choice for a Lagrangian structure $\lagra$ in this case: use ${\symfo}$ to pair the copy of $\tanga{\deri{A}{}{0}}$ in degree $1$ with $\ubun$, and $\nesplit{E}$ with the copy of $\posplit{E}$ in degree $2$.

Notice that the obvious inclusion $\lalge\hookrightarrow\widetilde{\lalge}$ is a weak equivalence of dg $\dera{A}$-modules, and, moreover, it pulls back $\lagra$ on $\widetilde{\lalge}$ to $0$ on $\lalge$. So $(\lalge,0)\rightarrow(\widetilde{\lalge},\lagra)$ is an equivalence between isotropic distributions. It follows, in particular, that $(\lalge,0)$ is a Lagrangian distribution according to Def.\@ \ref{DefLagrangian}.\end{example}
We finish this section with a description of an isotropic distribution on a derived scheme, written in terms of a derived affine atlas. Recall that having a derived atlas as described in Section \ref{SectionDistributions} we have the notions of integrable distributions (Def.\@ \ref{DefAtlasDistribution}) and shifted symplectic forms (Def.\@ \ref{DefSymAt}) on this atlas. 

\begin{definition} Let $\dersch$ be a derived scheme equipped with a derived affine atlas $\big\{\spec{\deri{A}{j}\bullet}\big\}_{j\in J}$. Suppose we are given an $n$-shifted symplectic form and an integrable distribution on this atlas:
	\begin{equation*}\symstra=\big\{\symstra_{j_0,\ldots,j_k}\;|\;k\geq 0,\,j_0,\ldots,j_k\in J\big\},\end{equation*}
	\begin{equation*}(\grami{\bullet}{\bullet},\anchord,\conne{})=\big\{\gramii{\bullet}{\bullet}{j_0,\ldots,j_k},\anchord_{j_0,\ldots,j_k},\conne{\makebo{j}{k}{s}}\,|\,k\geq 0,\,j_0,\ldots,j_k\in J,0\leq s\leq k\big\},\end{equation*}
\emph{An isotropic structure on this atlas} is given by 
	\begin{equation*}\lagras=\big\{\lagras_{j_0,\ldots,j_k}\in\negacy{n-k-3}{2}\big(\gramii{\bullet}{\bullet}{j_0,\ldots,j_k}\big)\,|\,k\geq 0,\,j_0,\ldots,j_k\in J\big\}\end{equation*}
s.t.\@ $\forall j\in J$ $(\homdi{}+t\cirdi)(\lagras_j)=\anchord_{j}(\symstra_j)$ i.e.\@ $(\anchord_j,\lagras_j)$ is an isotropic distribution on $(\spec{\deri{A}{j}\bullet},\symstra_j)$, and $\forall k\geq 1$
	\begin{equation*}(\homdi{}+t\cirdi)(\lagras_{j_0,\ldots,j_k})=\underset{0\leq s\leq k}\sum(-1)^s\conne{\makebo{j}{k}{s}}(\lagras_{j_0,\ldots,\widehat{j_s},\ldots,j_k})+
	(-1)^{k+1}\anchord_{j_0,\ldots,j_k}(\symstra_{j_0,\ldots,j_k}).
	\footnote{This particular choice of signs is due to us putting $0\in\negacy{n-\bullet}{2}\big(\gramii{\bullet}{\bullet}{j_0,\ldots,j_k}\big)$ 
	always as the last vertex in the $k+1$-simplex of elements of $\negacy{n-\bullet}{2}\big(\gramii{\bullet}{\bullet}{j_0,\ldots,j_k}\big)$.}\end{equation*}
Such distribution will be called \emph{Lagrangian}, if $\forall j\in J$ the isotropic distribution $(\anchord_j,\lagras_j)$ on $(\spec{\deri{A}{j}\bullet},\symstra_j)$ is Lagrangian.\end{definition}

\subsection{Sheaves of symplectic and isotropic structures}

Just as different atlases can define the same derived scheme, so are the notions of integrable distributions, or shifted symplectic forms, or isotropic structures on an atlas contain too much information, that needs to be discarded by an appropriately defined equivalence relation. To begin defining this relation we notice that it is easy to pull back all three kinds of structure. 

Explicitly we can view an indexing set $J$ as vertices of a simplex $\Delta(J)$, and then having a surjective map $\pi\colon J\rightarrow K$ we have a morphism of simplices $\pi\colon\Delta(J)\rightarrow\Delta(K)$. Suppose that we have a natural transformation
	\begin{equation*}\forall\sigma\in\Delta(J)\quad\pi_\sigma\colon\spec{\deri{A}{\sigma}\bullet}\longrightarrow\spec{\deri{B}{\pi(\sigma)}\bullet}\end{equation*}
in the category of affine derived schemes over $\dersch$. Then an integrable distribution, or a symplectic structure, or an isotropic structure on $\big\{\spec{\deri{B}{\pi(\sigma)}\bullet}\big\}$ pulls back to the same kind of structure on $\spec{\deri{A}{\sigma}\bullet}$. It is clear that the result and the initial data describe the same object on $\dersch$.

Here is a simple example. Given an atlas indexed by $J$, we can always divide each chart into smaller derived affine pieces obtaining a finer atlas indexed by $J'$ with the obvious surjection $J'\rightarrow J$. Pulling back, i.e.\@ restricting distributions, symplectic forms or isotropic structures to the smaller pieces we obtain equivalent structures of the same kind. It is immediate to see that declaring shifted symplectic structures on two atlases equivalent, if there is a common subdivision where the two structures are equal, gives us an equivalence relation. Similarly for isotropic distributions.

Subdivision is particularly easy because making a derived affine open chart smaller involves inverting an actual function of degree $0$. Of course there are more possibilities for equivalence obtained not by an actual localization but by a weak equivalence of derived affine charts. The following definition is meant to encode a stronger equivalence relation.

\begin{definition}\label{DefEquivalence} Let $\big\{\spec{\deri{A}{j}{\bullet}}\big\}_{j\in J}$, $\big\{\spec{\deri{B}{k}{\bullet}}\big\}_{k\in K}$ be two atlases on a derived scheme $\dersch$, and let $\big\{\symstra_{j_0,\ldots,j_m}\big\}_{m\geq 0}$, $\big\{\symstra'_{k_0,\ldots,k_n}\big\}_{n\geq 0}$ be $n$-shifted symplectic forms on $\big\{\spec{\deri{A}{j}{\bullet}}\big\}_{j\in J}$, $\big\{\spec{\deri{B}{k}{\bullet}}\big\}_{k\in K}$ respectively. We will say that these \emph{two shifted symplectic structures are equivalent}, if there are sub-divisions $\big\{\spec{\deri{\wh{A}}{j'}{\bullet}}\big\}_{j'\in J'}$, $\big\{\spec{\deri{\wh{B}}{k'}{\bullet}}\big\}_{k'\in K'}$ of the two atlases, both of which are contained in a third atlas $\big\{\spec{\deri{C}{l}{\bullet}}\big\}_{l\in L}$ that is equipped with a $n$-shifted symplectic structure $\big\{\symstra''_{l_0,\ldots,l_r}\big\}_{r\geq 0}$ s.t.\@
	\begin{equation*}\big\{\symstra''_{l_0,\ldots,l_r}\big\}\Big|_{\big\{\spec{\deri{\wh{A}}{j'}{\bullet}}\big\}}=\big\{\symstra_{j'_0,\ldots,j'_m}\big\},\quad
	\big\{\symstra''_{l_0,\ldots,l_r}\big\}\Big|_{\big\{\spec{\deri{\wh{B}}{k'}{\bullet}}\big\}}=\big\{\symstra'_{k'_0,\ldots,k'_n}\big\}.\end{equation*}
Analogously we define \emph{equivalence between integrable distributions on atlases and between isotropic structures on atlases}.\end{definition}
\begin{proposition} The relation defined in Def.\@ \ref{DefEquivalence} is an equivalence relation.\end{proposition}
\begin{proof} The only non-obvious property is transitivity. Let  $\big\{\symstra^1_{j^1_0,\ldots,j^1_{m_1}}\big\}_{m_1\geq 0}$, $\big\{\symstra^2_{j^2_0,\ldots,j^2_{m_2}}\big\}_{m_2\geq 0}$, $\big\{\symstra^3_{j^3_0,\ldots,j^3_{m_3}}\big\}_{m_3\geq 0}$ be $n$-shifted symplectic forms on atlases $\big\{\spec{\deri{A}{j^1}{\bullet}}\big\}_{j^1\in J_1}$, $\big\{\spec{\deri{B}{j^2}{\bullet}}\big\}_{j^2\in J_2}$, $\big\{\spec{\deri{C}{j^3}{\bullet}}\big\}_{j^3\in J_3}$. Suppose that the first is equivalent to the second, and the second is equivalent to the third according to Def.\@ \ref{DefEquivalence}, i.e.\@ we have other two atlases $\big\{\spec{\deri{D}{j^{12}}{\bullet}}\big\}_{j^{12}\in J_{12}}$, $\big\{\spec{\deri{E}{j^{23}}{\bullet}}\big\}_{j^{23}\in J_{23}}$ containing sub-divisions of the first three, together with $n$-shifted symplectic structures $\big\{\symstra^{12}_{j^{12}_0,\ldots,j^{12}_{m_{12}}}\big\}_{m_{12}\geq 0}$, $\big\{\symstra^{23}_{j^{23}_0,\ldots,j^{23}_{m_{23}}}\big\}_{m_{23}\geq 0}$ with the prescribed restrictions.

The $5$ atlases give $5$ open covers of the space of classical points $\unspa$. Choosing a common refinement of these $5$ open covers and localizing the corresponding derived affine schemes, we can assume that $J_2\subseteq J_{12}$, $J_2\subseteq J_{23}$ and we have surjective maps $\pi_{12}\colon J_{12}\rightarrow J_2$, $\pi_{23}\colon J_{23}\rightarrow J_2$ s.t.\@ $\forall j^2\in J_2$ $\spec{\deri{B}{j^2}\bullet}$ is weakly equivalent over $\dersch$ to any chart indexed by an element in the fibers of $\pi_{12}$, $\pi_{23}$ over $j^2$. 

We would like to construct another atlas $\big\{\spec{\deri{F}{j}\bullet}\big\}$ s.t.\@ it contains subdivisions of $\big\{\spec{\deri{D}{j^{12}}{\bullet}}\big\}_{j^{12}\in J_{12}}$, $\big\{\spec{\deri{E}{j^{23}}{\bullet}}\big\}_{j^{23}\in J_{23}}$, and equip $\big\{\spec{\deri{F}{j}\bullet}\big\}$ with a shifted symplectic structure that restricts to $\big\{\symstra^{12}_{j^{12}_0,\ldots,j^{12}_{m_{12}}}\big\}_{m_{12}\geq 0}$, $\big\{\symstra^{23}_{j^{23}_0,\ldots,j^{23}_{m_{23}}}\big\}_{m_{23}\geq 0}$. The indexing set for this sixth atlas will be
	\begin{equation}\label{Indexing}J:=J_{12}\sqcup J_{23}.\end{equation}
We need to specify intersections of charts that belong to different parts in the decomposition (\ref{Indexing}). If we view $J_2$ as a subset of $J_{12}$, we have $\pi_{23}\colon J_{23}\rightarrow J_{12}$, and vice versa. For $k\geq 1$ let $\{j_0,\ldots,j_k\}\subseteq J$ s.t.\@ $\{j_0,\ldots,j_{k'}\}\subseteq J_{12}$ and $\{j_{k'+1},\ldots,j_k\}\subseteq J_{23}$ with $k'\neq k$. We define
	\begin{equation*}\spec{\deri{F}{j_0,\ldots,j_k}\bullet}:=\underset{0\leq s\leq s'}\bigcap\spec{\deri{D}{j_s}\bullet}\cap\underset{k'+1\leq s\leq k}\bigcap\spec{\deri{E}{j_s}\bullet}
	\cap\underset{k'+1\leq s\leq k}\bigcap\spec{\deri{D}{\pi_{23}(j_s)}\bullet}.\end{equation*}
This gives us an atlas indexed by $J$. Now we need to define a shifted symplectic structure on this atlas. On intersections of charts indexed by $J_{12}$ or $J_{23}$ separately we already have it defined. Let $\{j_0,\ldots,j_k\}\subseteq J$, s.t.\@ $j_0\in J_{12}$ and $\{j_1,\ldots, j_k\}\subseteq J_{23}$. Consider the simplicial complex $\{0,1\}\times\{j_1,\ldots, j_k\}$. It is a prism of dimension $k$ and it is naturally decomposed into a union of $k$-simplices. We think of the top of this prism as being the simplex $\{\pi_{23}(j_1),\ldots,\pi_{23}(j_k)\}\subseteq J_{23}$ and the bottom as being $\{j_1,\ldots,j_k\}$.

 The orientations on $\{0,1\}\times\{j_1\}$ and $\{1\}\times\{j_1,\ldots, j_k\}$ given by ordering defines a orientation of the prism, and we assign a sign $\epsilon_{j'_0,\ldots, j'_k}$ to each $k$-simplex $\{j'_0,\ldots,j'_k\}\subset\{0,1\}\times\{j_1,\ldots,j_k\}$ according to whether this orientation agrees with the ordering or not. Then, suppressing the notation for pullbacks of differential forms we define
	\begin{equation*}\symstra_{j_0,\ldots,j_k}:=\symstra_{j_0,\pi_{23}(j_1),\ldots,\pi_{23}(j_k)}+\sum\epsilon_{j'_0,\ldots,j'_k}\symstra_{j'_0,\ldots,j'_k},\end{equation*}
where the sum runs over all $k$-simplices in the prism. If $\{j_0,\ldots,j_k\}\subseteq J$ is such that $\{j_0,\ldots,j_{k'}\}\subseteq J_{12}$ with $k'\geq 1$ we define
	\begin{equation*}\symstra_{j_0,\ldots,j_k}:=\symstra_{j_0,\ldots,j_{k'},\pi_{23}(j_{k'+1}),\ldots,\pi_{23}(j_k)}.\end{equation*}
It is tedious but straightforward to check that this gives us a shifted symplectic structure on the atlas indexed by $J$.

To extend distributions from $J_{12}$- and $J_{23}$-atlases to the $J$-atlas it is not enough just to pull-back, one needs also to push-forward. We have 
	\begin{equation}\label{left}\rahm{\deri{D}{j_0,\ldots,j_{k'},\pi_{23}(j_{k'+1}),\ldots,\pi_{23}(j_k)}\bullet}\bullet\longrightarrow\gramii{\bullet}{\bullet}{j_0,\ldots,j_{k'},\pi_{23}(j_{k'+1}),\ldots,\pi_{23}(j_k)},\end{equation}
	\begin{equation}\label{right}\rahm{\deri{E}{\pi_{23}(j_{k'+1}),\ldots,\pi_{23}(j_k),j_{k'+1},\ldots,j'_k}\bullet}\bullet\longrightarrow\gramii{\bullet}{\bullet}{\pi_{23}(j_{k'+1}),\ldots,\pi_{23}(j_k),j_{k'+1},\ldots,j'_k}.\end{equation}
Suppressing the notation for pullbacks of bundles and forms we define 
	\begin{equation*}\rahm{\deri{F}{j_0,\ldots,j_k}\bullet}{\bullet}\longrightarrow\gramii\bullet\bullet{j_0,\ldots,j_k}\end{equation*}
as the push-out of (\ref{left}) and (\ref{right}) under 
	\begin{equation*}\rahm{\deri{D}{\pi_{23}(j_{k'+1}),\ldots,\pi_{23}(j_k)}\bullet}\bullet\longrightarrow\gramii{\bullet}{\bullet}{\pi_{23}(j_{k'+1}),\ldots,\pi_{23}(j_k)}.\end{equation*}
Once the integrable distributions are extended to the whole $J$-atlas, an extension of the isotropic structures proceeds analogously to the extension of symplectic forms.\end{proof}%

\begin{corollary} Let $\dersch$ be a derived scheme and let $\unspa$ be the underlying topological space. For any open subset $\opes\subseteq\unspa$ let $\symsh{}(\opes)$, $\ish{}(\opes)$ be the sets of equivalence classes of atlases on $\opes$ equipped with shifted symplectic forms and isotropic structures respectively. Then $\symsh{}$, $\ish{}$ are sheaves on $\unspa$. 

Similar statement holds, if we replace $\dersch$ with a derived manifold $\derm$ whose underlying space $\unspam$ of classical points carries the $\cinfty$-topology.\end{corollary}

We would like to compare sets of sections of these sheaves with the sets of symplectic or isotropic structures obtained by the standard techniques of stacks. According to \cite{PTVV13} Def.\@ 1.12 \emph{the space of $n$-shifted symplectic forms on $\dersch$} is the simplicial set $\mapp(\dersch,\forms{n})$ where $\forms{n}$ is the stack of closed $2$-forms of degree $n$ and the mapping space is computed in the category of derived stacks (relative to the \'etale topology). In particular for an affine derived scheme $\spec{\dera{A}}$ we get Def.\@ \ref{DefSymAffine}.

\begin{proposition} Let $\dersch$ be a derived scheme, and let $\symsh{n}$ be the sheaf of equivalence classes of $n$-shifted symplectic forms on $\unspa$ (Def.\@ \ref{DefEquivalence}). Then we have a natural bijection
	\begin{equation}\label{SheafToStack}\symsh{n}(\unspa)\overset\cong\longrightarrow\pi_0(\mapp(\dersch,\forms{n})).\end{equation}\end{proposition}
\begin{proof} Let $\{\spec{\deri{A}{j}\bullet}\}_{j\in J}$ be a derived affine atlas on $\dersch$. In particular this means that
	\begin{equation*}\dersch\simeq\underset{k\geq 0,\{j_0,\ldots,j_k\}\subseteq J}\hocolim\bigg(\underset{0\leq s\leq k}\bigcap\spec{\deri{A}{j_s}\bullet}\bigg),\end{equation*}
and correspondingly 
	\begin{equation}\label{MapForms}\mapp(\dersch,\forms{n})\simeq\underset{k\geq 0,\{j_0,\ldots,j_k\}\subseteq J}\holim\mapp(\spec{\deri{A}{j_0,\ldots,j_k}\bullet},\forms{n}).\end{equation}
Def.\@ \ref{DefSymAt} gives an explicit representation of a $0$-simplex in this homotopy limit of simplicial sets. Pulling back $0$-simplices in (\ref{MapForms}) over inclusion of atlases and sub-division of atlases clearly produces $0$-simplices in the same connected component. Hence we get the map in (\ref{SheafToStack}). Since every shifted symplectic form on $\dersch$ can be realized (up to homotopy) on any given atlas, it is clear that (\ref{SheafToStack}) is surjective. 

Two simplices in the same connected component of (\ref{MapForms}) are connected by a system of $1$-simplices, that are connected by $2$-simplices and so forth, with $k$-simplices being defined on $k$-fold intersections. Altogether this is nothing else but the structure of a shifted symplectic form on the disjoint union of $\{\spec{\deri{A}{j}\bullet}\}_{j\in J}$ with itself. So (\ref{SheafToStack}) is injective as well.\end{proof}%

\section{Strictification of Lagrangian distributions}\label{SectionStrict}

In this section we look at various levels of strictification of the notions considered above. Prop.\@ \ref{LocalD} is an example of strictification of symplectic forms. Our objective in this section is to prove a similar statement regarding Lagrangian distributions. The strict version of such distributions is given by a sub-complex of the tangent complex, having half the rank, s.t.\@ the symplectic form vanishes on the sub-complex. We will need to consider a somewhat relaxed variant, that we will call \emph{semi-strict}, where we require vanishing of the symplectic form only over the classical points. We start by introducing conditions that will allow us to prove the strictification result.

\bigskip

A general integrable distribution on $\spec{\dera{A}}$ is defined as a morphism $\rahm{\dera{A}}\bullet\longrightarrow\grami\bullet\bullet$. There are some conditions that this morphism should satisfy, but none require it to be surjective or put any vanishing restrictions on cohomology. We would like to introduce conditions of this kind.

\begin{definition}\label{DefPurelyDerived} Let $\dera{A}$ be a well presented dg algebra, an integrable distribution
	\begin{equation*}\anchord\colon\rahm{\dera{A}}\bullet\longrightarrow\grami\bullet\bullet,\end{equation*} 
will be called \emph{strict}, if $\anchord$ is surjective. A distribution would be called \emph{purely derived}, if $\homology{\geq 0}(\grami{\bullet}{1})=0$. An integrable distribution that is purely derived and simultaneously a derived foliation will be called \emph{purely derived foliation}.\end{definition}
\begin{remark} Since we have required that, forgetting $\cirdi$, an integrable distribution $\grami\bullet\bullet$ is a free dg $\dera{A}$-algebra generated by $\grami\bullet{1}$, it is immediately clear that the kernel of a surjective $\anchord\colon\drham\bullet{}{\dera{A}}\rightarrow\grami\bullet\bullet$ is a graded mixed ideal having generators in $\drham{1}{}{\dera{A}}$.\end{remark}
In terms of Lie--Rinehart algebras a strict distribution is just a dg sub-module of the tangent complex, that is closed with respect to the Lie bracket. In Def.\@ \ref{DefDerivedFoliation} we have singled out those distributions -- derived foliations -- that locally always have a strict presentation. Now we would like to see what simplification we get from the purely derived condition. Since we would like to incorporate isotropic structures we need first to define what we mean by an equivalence between isotropic distributions.

\begin{definition} We will say that two isotropic distributions
	\begin{equation*}\anchord_1\colon(\rahm{\dera{A}}\bullet,\symstra)\rightarrow(\gramii\bullet\bullet{1},\lagras_1),\quad
	\anchord_2\colon(\rahm{\dera{A}}\bullet,\symstra)\rightarrow(\gramii\bullet\bullet{2},\lagras_2)\end{equation*}
are \emph{equivalent}, if there are $\anchord_3\colon(\rahm{\dera{A}}\bullet,\symstra)\rightarrow(\gramii\bullet\bullet{3},\lagras_3)$, $\dima{1},\dima{2}$ defining an equivalence $\gramii\bullet\bullet{1}\sim\gramii\bullet\bullet{2}$ according to Def.\@ \ref{DefEquivalenceDist}, s.t.\@ $\dima{2}(\lagras_3)\sim\lagras_2$, $\dima{1}(\lagras_3)\sim\lagras_1$ according to Def.\@ \ref{DefIsoDist}.\end{definition}
\begin{proposition}\label{StriPro} Let $\dera{A}$ be a well presented dg algebra and let $\symstra$ be a homotopically closed $2$-form on $\spec{\dera{A}}$. Let $\anchor\colon(\lalge,\lagras)\longrightarrow(\tanga{{\dera{A}}},\symstra)$ be a purely derived strict distribution with an isotropic structure, s.t.\@ 
	\begin{equation}\label{SecondHomologyMap}\homology{\geq 2}(\lalge)\overset\cong\longrightarrow\homology{\geq 2}(\tanga{{\dera{A}}}).\end{equation} 
Then there is an equivalent strict distribution with an isotropic structure $\anchor'\colon(\lmin,\lagras')\rightarrow(\tanga{\dera{A}},\symstra)$, s.t.\@ $\lmin$ is a perfect dg $\dera{A}$-module and $\forall\pt\colon\dera{A}\rightarrow\mathbb C$ \begin{enumerate}[label=(\roman*)]
\item $\lmin^{\leq 0}|_\pt=0$,
\item $\lmin^{\geq 2}|_\pt\overset{\cong}\longrightarrow\tangde{{\dera{A}}}{\geq 2}|_\pt$,
\item $\lmin^1|_\pt\longrightarrow\tangde{{\dera{A}}}{1}|_{\pt}$ is injective.\end{enumerate}\end{proposition}
\begin{proof} As the distribution is strict, $\anchor\colon\lalge\rightarrow\tanga{{\dera{A}}}$ is a component-wise inclusion of perfect $\dera{A}$-modules, that are closed with respect to the Lie bracket on the tangent complex. We will construct now another sub-complex of $\tanga{\dera{A}}$, that is weakly equivalent over $\tanga{\dera{A}}$ to $\lalge$, closed with respect to the Lie bracket (for degree reasons), and satisfies the conditions (i)-(iii) above.

We denote by $\{G^k\}_{k\geq 0}$ a sequence of projective $\deri{A}{}{0}$-modules, s.t.\@ they generate the underlying graded $\undera{A}$-module of $\lalge$, with $G^k$ sitting in degree $k$. Since the distribution is purely derived, the dg $\dera{A}$-submodule of $\lalge$ generated by $G^0$ is acyclic. Let $\lalge_1$ be a complement of this dg sub-module in $\lalge$. For degree reasons $\lalge_1$ is a dg Lie sub-algebra of $\lalge$, and $\lalge_1\hookrightarrow\lalge$ is clearly a weak equivalence. Restricting $\lagras$ we obtain an equivalent isotropic distribution
	\begin{equation*}\anchor_1\colon(\lalge_1,\lagras_1)\longrightarrow(\tanga{{\dera{A}}},\symstra).\end{equation*}
By construction $\lalge_1$ has generators in degrees $\geq 1$, and $\lalge_1\rightarrow\tanga{{\dera{A}}}$ is injective, but it is not necessarily true that $\anchor_1$ is surjective in degrees $\geq 2$. We can choose generating bundles $\{G_1^k\}_{k\geq 1}$, $\{\ul{G}^k\}_{k\geq 1}$ for $\lalge_1$, $\tanga{{\dera{A}}}$ respectively, s.t.\@ for each $k\geq 1$ $\anchor_1$ realizes $G_1^k$ as a sub-bundle of $\ul{G}^k$. Using (\ref{SecondHomologyMap}) and starting from the top degree we can choose a decomposition
	\begin{equation*}\ul{G}^k\cong G_1^k\oplus B^k,\quad k\geq 2,\quad G_1^1\oplus B^1\subseteq\ul{G}^1\end{equation*}
where each $B^k$ is a projective $\deri{A}{}{0}$-module, $\homdi{}\colon B^k\rightarrow B^{k+1}$, and $(\underset{k\geq 1}\bigoplus B^k,\homdi{})$ is acyclic. Let $\lmin$ be the the direct sum of $\lalge_1$ and the dg submodule generated by $\{B^k\}_{k\geq 1}$. Then $\lmin$ is a sub-complex of $\tanga{\dera{A}}$ and it is closed with respect to the Lie bracket for degree reasons. In the obvious factorization
	\begin{equation*}\lalge_1\overset\simeq\longrightarrow\lmin\overset{\anchor'}\longrightarrow\tanga{{\dera{A}}},\end{equation*}
the first arrow is clearly a weak equivalence, and hence we can extend $\lagras$ to an equivalent isotropic structure $\lagras'$ on $\lmin$.\end{proof}%

Although the distribution constructed in Prop.\@ \ref{StriPro} is similar to the strict Lagrangian distributions considered in Example \ref{ExIso}, it is not necessarily true that $\anchord'(\symstra)=0$, even at the closed points. There is still one level of strictification left. We take care of it in the following proposition. For simplicity we formulate it for a particular kind of affine derived schemes: given a well presented $\dera{A}$, we will say that $\spec{\dera{A}}$  \emph{can be made flat}, if $\tanga{{\dera{A}}}$ has a graded vector subspace $V^*$ that generates $\tanga{{\dera{A}}}$ over $\dera{A}$ and s.t.\@ the Lie bracket vanishes on $V^*$. This can always be achieved by localization.

\begin{proposition}\label{TrivLa} Let $\dera{A}$ be a well presented dg algebra s.t.\@ $\spec{\dera{A}}$ can be made flat, and let $\symstra$ be a homotopically closed $2$-form of degree $-2$ on $\spec{\dera{A}}$, s.t.\@ $\forall\pt\colon\dera{A}\rightarrow\mathbb C$ $\symfo$ defines a perfect pairing between $\tanga{\deri{A}{}{0}}|_\pt$ and ${\tangde{\dera{A}}{2}|_\pt}$. Let $\anchor\colon(\lmin,\lagras)\rightarrow(\tanga{{\dera{A}}},\symstra)$ be an isotropic distribution satisfying conditions (i)-(iii) in Prop.\@ \ref{StriPro}. Then there is an equivalent isotropic distribution $\anchor'\colon(\lmin',\lagras')\rightarrow(\tanga{{\dera{A}}},\symstra)$ having properties (i)-(ii) above, and s.t.\@ $\anchord'(\symfo)=0$ modulo $\homdi{}(\deri{A}{}{-1})$. If in addition $\lmin$ is a derived foliation, $\lmin'$ would satisfy condition (iii) as well.\end{proposition}
\begin{proof} By assumption $\lmin\rightarrow\tanga{\spec{\dera{A}}}$ is an isomorphism in degrees $\geq 2$ and an inclusion in all other degrees, but $\symstra|_\lmin$ might not be $0$ on the nose. We proceed then by enlarging $\lmin$ to make the anchor surjective in each degree, and then carefully extracting a sub-complex where $\symstra$ vanishes.

The underlying graded $\undera{A}$-module of the tangent complex $\tanga{{\dera{A}}}$ is generated by projective $A^0$-modules: $\tanga{{\deri{A}{}{0}}}$, $\tosplit{E}$ and $\{\ubun^k\}$ in degrees $0$, $1$ and $k\geq 2$ respectively. Assumptions (i)-(iii) on $\lmin$ imply that the underlying graded $\undera{A}$-module is generated by projective $\deri{A}{}{0}$-modules in degrees $1$ and $\geq 2$, that we will denote by $\lhalf{E}$ and $\{\ubun^k\}$ respectively. Because of (ii) we can use $\{\ubun^k\}$ to denote generators both for $\lmin$ and $\tanga{{\dera{A}}}$. Using (iii) we can assume that $\lhalf{E}$ is a sub-bundle of $\tosplit{E}$, and we choose a complement $\rhalf{E}$ of $\lhalf{E}$ in $\tosplit{E}$. 

\smallskip

The quotient $\tanga{{\dera{A}}}/\lmin$ clearly has $\tanga{{\deri{A}{}0}}$ and $\tosplit{E}/\lhalf{E}$ as the generating bundles. Let us denote by $C^\bullet$ a surjective homotopy kernel of $\tanga{{\dera{A}}}\rightarrow\tanga{{\dera{A}}}/\lmin$. Explicitly we can write generating bundles of $C^\bullet$ as follows
	\begin{equation*}\tanga{{\deri{A}{}0}},\quad \lhalf{E}\oplus\rhalf{E}\oplus\suspensi{\tanga{{\deri{A}{}0}}}{-1},\quad\ubun^2\oplus\suspensi{\tosplit{E}/\lhalf{E}}{-1},\quad\ubun^3,\ldots\end{equation*}
in degrees $0$, $1$, $2$, $3$ and so on. The canonical morphism ${C^\bullet}\rightarrow\tanga{{\dera{A}}}$ maps $\suspensi{\tanga{{\deri{A}{}{0}}}}{-1}$, $\suspensi{\tosplit{E}/\lhalf{E}}{-1}$ to $0$. Choosing a graded vector subspace of $\tanga{\deri{A}{}{0}}\oplus\tosplit{E}\oplus\big(\underset{k\geq 2}\bigoplus\ubun^k\big)$ with vanishing Lie bracket that generates the entire complex we obtain a generating subspace also for ${C^\bullet}$, and we can define a Lie--Rinehart structure on ${C^\bullet}$ by using the anchor map to extend the $0$ bracket. Clearly $\lmin\rightarrow\tanga{{\dera{A}}}$ factors through the quasi-isomorphism $\lmin\hookrightarrow{C^\bullet}$, i.e.\@ we have an equivalence of integrable distributions.

\smallskip

Having chosen $\rhalf{E}$ we can view $\lagras$ as defined on all of ${C^\bullet}$, and since the complement of $\lmin$ in ${C^\bullet}$ is contractible, we can extend this $\lagras$ to an isotropic structure $\lagras'$ on ${C^\bullet}$. This extension is unique up to a $\homdi{}+t\cirdi$-coboundary. As usual we denote by $\lagra'$ the free term in $\lagras'$. Our assumptions on $\symfo$ imply that $\lagra'$ defines a perfect pairing between $\suspensi{\tanga{{\deri{A}{}{0}}}}{-1}$ and ${\ubun^2}$. On the other hand $\homdi{}(\rhalf{E})$ has a trivial intersection with the dg sub-module generated by $\{\ubun^k\}_{k\geq 2}$. Let $(\ubun^2)^\perp\subseteq C^1$ be the orthogonal complement with respect to $\lagra'$. Since $\homdi{}(\suspensi{\tanga{{\deri{A}{}{0}}}}{-1})=0$ we can choose 
	\begin{equation*}K\subseteq(\ubun^2)^\perp\cap\big(\lhalf{E}\oplus\rhalf{E}\oplus\suspensi{\tanga{{\deri{A}{}0}}}{-1}\big),\end{equation*} 
s.t.\@ $\homdi{}\colon K\overset\cong\longrightarrow\homdi{}(\rhalf{E})$. Now let $M\subseteq \lhalf{E}\oplus\rhalf{E}\oplus\tanga{{\deri{A}{}0}}$ be any complement of $K$, s.t.\@ $\homdi{}(M)$ belongs to the dg sub-module generated by $\{\ubun^k\}_{k\geq 2}$. By construction $\lagra'$ maps $\ubun^2$ injectively into the dual of $M$. Let $N$ be the orthogonal complement of $\ubun^2$ in $M$ with respect to $\lagra'$, and let $\lmin'\subseteq{C^\bullet}$ be the dg $\dera{A}$-submodule generated by $N$ and $\{\ubun^k\}_{k\geq 2}$.

By construction $\lmin'$ is a complement to the dg $\dera{A}$-submodule of ${C^\bullet}$ generated by $\tanga{\deri{A}{}0}$ and $K$. The latter dg sub-module is acyclic, hence $\lmin'\hookrightarrow{C^\bullet}$ is an equivalence of integrable distributions. Restricting $\lagras'$ to $\lmin'$ we obtain an isotropic structure. By definition $\deg\lagra'=-3$ and hence
	\begin{equation*}\lagra'\in\deri{A}{}{-1}\underset{\deri{A}{}{0}}\otimes\dumo{N}\underset{\deri{A}{}{0}}\otimes\dumo{N},\end{equation*}
i.e.\@ $\anchord'(\symstra)\in\homdi{}(\deri{A}{}{-1})\underset{\deri{A}{}{0}}\otimes\dumo{N}\underset{\deri{A}{}{0}}\otimes\dumo{N}$. In particular $\anchord'(\symstra)=0$ modulo $\homdi{}(\deri{A}{}{-1})$.

\smallskip

If $\lmin$ is a derived foliation, any strict distribution, that is equivalent to $\lmin$ and satisfies conditions (i)-(ii), has to satisfy condition (iii). Otherwise any non-trivial element in the kernel of the anchor would give a non-trivial cohomology class, that is mapped to $0$ in the cohomology of $\tanga{\dera{A}}$.
\end{proof}%

\begin{definition}\label{DefStrictIso} Let $\dera{A}$ be a well presented dg algebra and let $\symstra$ be a homotopically closed $2$-form of degree $-2$ on $\spec{\dera{A}}$. An isotropic distribution 
	\begin{equation*}\anchord\colon(\rahm{\dera{A}}\bullet,\symstra)\longrightarrow(\grami\bullet\bullet,\lagras)\end{equation*} 
is \emph{semi-strict}, if it satisfies conditions (i)-(iii) of Prop.\@ \ref{StriPro} and 
	\begin{equation*}\anchord(\symstra)\equiv 0\mod\homdi{}(\deri{A}{}{-1}).\end{equation*}
\end{definition}
\begin{remark}\label{LocallySemistrict} Around any closed point we can choose a minimal representative of the dg algebra. Since on minimal dg algebras symplectic forms define perfect pairings between the corresponding generating bundles of the tangent complex, Propositions \ref{StriPro}, \ref{TrivLa} imply that locally any purely derived foliation on a derived $1$-stack with a $-2$-shifted symplectic form can be made semi-strict.\end{remark}
\begin{remark}\label{StrictProperty} For an isotropic distribution $(\grami\bullet\bullet,\lagras)$ to be semi-strict is a property of only $\anchord\colon\rahm{\dera{A}}\bullet\rightarrow\grami\bullet\bullet$, it does not depend on the choice of $\lagras$.\end{remark}

It would be helpful to have a characterization of Lagrangian distributions that is similar to Remark \ref{StrictProperty}, i.e.\@ only in terms of the distribution itself. It is immediately clear that we need to have a condition involving ranks.

\begin{definition} Let $\dera{A}$ be a dg algebra, and let $\lmin$ be a perfect dg $\dera{A}$-module. \emph{The rank of $\lmin$} at $\pt\colon\dera{A}\rightarrow\mathbb C$ is the Euler characteristic of $\lmin|_\pt$.\end{definition}
We will usually assume that the rank of $\lmin$ is the same at all closed points, and call it simply the rank of $\lmin$.  It is clear that ranks of Lagrangian distributions are always half of those of the tangent complexes. 

\begin{proposition}\label{Lagrangian} Let $\symfo$ be a strict $-2$-shifted symplectic form on $\spec{\dera{A}}$, and let $\anchor\colon(\lmin,\lagras)\rightarrow(\tanga{{\dera{A}}},\symfo)$ be a purely derived semi-strict isotropic distribution with $\rk{\lmin}{}=\frac{1}{2}\rk{\tanga{{\dera{A}}}}{}$. Then $(\lmin,\lagras)$ is a Lagrangian distribution.\end{proposition}
\begin{proof} Since $\symfo$ is a strict $-2$-shifted symplectic form and $\lhalf{E}\hookrightarrow\tosplit{E}$ is injective, the assumption on the rank of $\lmin$ immediately implies that $\rk{\lhalf{E}}{\deri{A}{}{0}}=\frac{1}{2}\rk{\tosplit{E}}{\deri{A}{}{0}}$. Since we assume that $\symfo|_{\lhalf{E}}=0$ over closed points it follows that $\symfo$ defines a perfect pairing between $\lhalf{E}$ and any complement $\rhalf{E}$. Using this complement we can enlarge $\lmin$ to an equivalent distribution with a surjective anchor. In degree $2$ this distribution is $\ubun\oplus\tosplit{E}/\lhalf{E}$, and $\lagra$ has to give a perfect pairing between $\homdi{}(\rhalf{E})$ and $\lhalf{E}$. Similarly for the copy of $\tanga{\deri{A}{}{0}}$ in degree $1$ and $\ubun$. But $\homdi{}(\rhalf{E})$ and $\tanga{\deri{A}{}{0}}$ coincide with the kernel of the anchor at $\pt$. Therefore in a neighbourhood of $\pt$ $\lagra$ defines an isomorphism between kernel of the anchor and the annihilator of the complex $\tanga{\deri{A}{}{0}}\rightarrow\homdi{}(\tanga{\deri{A}{}{0}})\oplus\rhalf{E}\rightarrow\homdi{}(\rhalf{E})$.\end{proof}

Since rank of a distribution is independent of the isotropic structure we might choose, we immediately have the following useful fact.

\begin{theorem} Let $\dera{A}$ be well presented and let $\symstra$ be a $-2$-shifted symplectic form on $\spec{\dera{A}}$. Let $\anchord\colon\rahm{\dera{A}}\bullet\rightarrow\grami\bullet\bullet$ be a purely derived foliation, s.t.\@  $\rk{\drham{1}{}{\dera{A}}}{}=2\rk{\grami\bullet{1}}{}$. If $\anchord(\symstra)$ is a $\homdi{}+t\cirdi$-coboundary, any isotropic structure on $\grami\bullet\bullet$ is necessarily Lagrangian.\end{theorem}
\begin{proof} The property of being a Lagrangian distribution is invariant with respect to weak equivalences of distributions and dg algebras. Computation of the kernel of the anchor map commutes with localization, and the Lagrangian property involves only the free term of the isotropic structure, in particular it does not involve the de Rham differential. Therefore to check that a given isotropic structure is Lagrangian it is enough to do so locally.

Locally we can switch to a chart where $\symstra$ is equivalent to a strict symplectic form $\symfo$, and since the distribution is a purely derived foliation we can find a semi-strict realization for it on this chart (Rem.\@ \ref{LocallySemistrict}). Now using Prop.\@ \ref{Lagrangian} we obtain the claim.\end{proof}%

This theorem simplifies construction of Lagrangian distributions. We need just to construct integrable distributions satisfying some cohomological conditions, and s.t.\@ the symplectic form becomes homotopically trivial there. Then any choice of an isotropic structure is automatically Lagrangian. In particular when we pull back our distributions from one chart to another and try to glue different pull-backs, it is enough to glue the integrable distributions.

\section{Global Lagrangian distributions over $\mathbb R$}\label{SectionGlobalLagrangian}

In this section we make the switch from geometry over $\mathbb C$ to geometry over $\mathbb R$. Suppose we have a derived scheme $\dersch$ over $\mathbb C$, then, at least locally on the space of classical points $\unspa$, $\dersch$ is given as a usual smooth scheme $\zeropart{S}$ together with a sequence of bundles on $\zeropart{S}$, morphisms between these bundles and their sections. E.g.\@ in the affine case $\zeropart{S}=\spec{\deri{A}{}{0}}$, and the bundles are given by the components of $\dera{A}$ in negative degrees.

Let $\mani$ be the subspace of closed points in $\zeropart{S}$. Then $\mani$ is a complex manifold, and we can consider a larger class of functions: 
	\begin{equation*}\cring{A}{0}:=\cinfty(\mani,\mathbb R).\end{equation*} 
Correspondingly we view the bundles on $\zeropart{S}$ as $\cinfty$-bundles on $\mani$, and thus obtain a dg manifold $\derm$. If $\dersch$ is equipped with a $n$-shifted symplectic structure $\symstra$, $\derm$ inherits two structures: 
	\begin{equation*}\refo{\symstra}:=\frac{\symstra+\conju{\symstra}}{2},\quad\imfo{\symstra}:=\frac{\symstra-\conju{\symstra}}{2 i}.\end{equation*}
Both are $\mathbb R$-valued $n$-shifted symplectic structures on $\derm$, and if $\symstra$ is strict, so are $\refo{\symstra}$ and $\imfo{\symstra}$. Following the analogy with gauge theory we will treat these $\mathbb R$-valued symplectic forms differently, as will be immediately apparent.

\begin{definition}\label{DefStriNe} Let $(\dersch,\symfo)$ be a derived scheme with a strict $\mathbb C$-valued $-2$-shifted symplectic structure. Let $(\derm,\imfo{\symfo},\refo{\symfo})$ be the underlying dg manifold with $\mathbb R$-valued symplectic structures. Let $\negsplit{E}{}\subseteq\tosplit{E}$ be a maximally isotropic sub-bundle with respect to $\imfo{\symfo}$, and let $\negta{\tanga{\cring{A}{\bullet}}}\subseteq\tanga{\cring{A}{\bullet}}$ be the dg sub-module generated by $\nesplit{E}\oplus\tangde{\cring{A}\bullet}{2}$. This is a strict Lagrangian distribution with respect to $\imfo{\symfo}$, and it will be called {\it strictly negative}, if $\refo{\symfo}$ is negative definite on $\nesplit{E}$.\end{definition}
In order to build Lagrangian distributions, that are globally defined, we need to be able to glue distributions coming from different charts. Having allowed $\cinfty$-functions we can make use of partitions of unity, and negative definiteness with respect to $\refo{\symfo}$ will allow us to do just that. The following proposition (taken from \cite{BoJ13}) provides the basis for this procedure.

\begin{proposition}\label{StrictInterpolating} Let $\negsplit{E}{0},\negsplit{E}{1}\subseteq\tosplit{E}$ be two maximally isotropic sub-bundles with respect to $\imfo{\symfo}$, s.t.\@ $\refo{\symfo}$ is negative definite on $\negsplit{E}{0}$, $\negsplit{E}{1}$. There is a smooth family $\negsplit{E}{t}$ $t\in[0,1]$ of maximally isotropic sub-bundles of $\tosplit{E}$ with respect to $\imfo{\symfo}$ interpolating between $\negsplit{E}{0}$ and $\negsplit{E}{1}$, s.t.\@ $\refo{\symfo}|_{\negsplit{E}{t}}$ is negative definite.\end{proposition}
\begin{proof} The splitting $\tosplit{E}=\nesplit{E_1}\oplus\, i\,\nesplit{E_1}$ gives us a projection $\pi\colon\nesplit{E_1}\rightarrow\nesplit{E_1}$, and then a $[0,1]$-family of maps 
	\begin{equation*}\pi_t:=\iota-t(\iota-\pi)\colon\nesplit{E_0}\longrightarrow\tosplit{E},\end{equation*}
where $\iota\colon\nesplit{E_0}\hookrightarrow\tosplit{E}$ is the inclusion. Obviously for each $t\in [0,1]$ the image of $\pi_t$ is isotropic with respect to $\imfo{\symfo}$. Moreover, $\pi_t$ is injective for each such $t$. Indeed, let $e$ be a local section of $\nesplit{E_0}$ s.t.\@
	\begin{equation*}\refo{\symfo}(e,e)=\refo{\symfo}(\pi(e),\pi(e))+\refo{\symfo}(e-\pi(e),e-\pi(e))<0.\end{equation*}
Then, since $\refo{\symfo}$ is negative definite on $\nesplit{E_0}$, for $0\leq t\leq 1$ we have
	\begin{equation*}\refo{\symfo}(\pi_t(e),\pi_t(e))=
	\refo{\symfo}(\pi(e),\pi(e))+t^2\refo{\symfo}(e-\pi(e),e-\pi(e))<0.\end{equation*}
So $\pi_t(e)\neq 0$ and we have shown, in addition, that $\refo{\symfo}$ is negative definite on the entire family.\end{proof}%

The family of strictly negative Lagrangian distributions on $\derm$ obtained from the image of $\pi_t$ $t\in[0,1]$ will be denoted by $\negtat{\tanga{\cring{A}{\bullet}}}{t}$. Now we can state the gluing argument (again taken from \cite{BoJ13}).

\begin{proposition}\label{StrictGluing} Let $U\subseteq\mani=\spec{\cring{A}{0}}$ be an open subset, and $V\subseteq U$, that is closed in $\mani$. Let $\negsplit{E}{1}$, $\negsplit{E}{0}$ be some maximally isotropic sub-bundles of $\tosplit{E}$ with respect to $\imfo{\symfo}$, defined on $\mani$ and $U$ respectively, s.t.\@ $\refo{\symfo}$ is positive definite on $\negsplit{E}{1}$ and on $\negsplit{E}{0}$. 

There is a sub-bundle $\negsplit{E}{{01}}$ defined on all of $\mani$, s.t.\@ it is maximally isotropic with respect to $\imfo{\symfo}$, restriction of $\refo{\symfo}$ to it is positive definite, and 
	\begin{equation*}\negsplit{E}{{01}}|_{\mani\setminus U}=\negsplit{E}{1}|_{\mani\setminus U},\quad\negsplit{E}{{01}}|_{V}=\negsplit{E}{0}|_{V}.\end{equation*}\end{proposition}
\begin{proof} Let $f\in\cring{A}{0}$ be s.t.\@ $f=1$ in a neighborhood of $V$ and $f=0$ in a neighborhood of $\mani\setminus U$. Let $W$ be the closure of the locus where $f\neq 0,1$. Both $\nesplit{E_0}$ and $\nesplit{E_1}$ are defined over $W$ and we can use the values of $f$ instead of $t$ in the family $\negtat{\tang{\cring{A}{\bullet}}{\mathbb R}|_W}{t}$ that interpolates from $\nesplit{E_0}$ to $\nesplit{E_1}$.\end{proof}%

In addition to the possibility of gluing, strictly negative Lagrangian distributions on $(\derm,\imfo{\symfo},\refo{\symfo})$ have very useful cohomological properties, implying that such distributions are in fact derived foliations.

\begin{proposition} Let $\anchor\colon\lmin\rightarrow\tanga{\derm}$ be a strictly negative Lagrangian distribution. Then $\anchor$ is injective on cohomology at each closed point.\end{proposition}
\begin{proof} As $\lmin$ has generators only in degrees $1$ and $2$, every cocycle in degree $1$ is non-trivial. Let $e$ be a section of $\tosplit{E}$ that is a coboundary. Then $\symfo(e,e)=0$, in particular $\refo{\symfo}(e,e)=0$. Thus $e\in\lhalf{E}\Rightarrow e=0$, i.e.\@ $\anchor$ is injective on degree $1$ cohomology. Since the image of $\anchor$ on degree $1$ cohomology has to be no more than half of all the degree $1$ cohomology, counting ranks we conclude that every coboundary in $\ubun$ lies in the image of $\anchor$.\end{proof}%

In order to be able to use the results of previous sections we need to be sure that after interpolating our integrable distributions still have the important property of being derived foliations. This is not difficult to see.

\begin{proposition} Let $\negsplit{E}{0},\negsplit{E}{1}\subseteq\tosplit{E}$ be two maximally isotropic sub-bundles with respect to $\imfo{\symfo}$, s.t.\@ $\refo{\symfo}$ is negative definite on $\negsplit{E}{0}$, $\negsplit{E}{1}$. Let $\negsplit{E}{t}$ $t\in[0,1]$ be the family of maximally isotropic sub-bundles of $\tosplit{E}$ with respect to $\imfo{\symfo}$ interpolating between $\negsplit{E}{0}$ and $\negsplit{E}{1}$ as in Prop.\@ \ref{StrictInterpolating}. If $\negsplit{E}{0}$, $\negsplit{E}{1}$ are derived foliations, the same would be true for the interpolating family.\end{proposition}
\begin{proof} Let $\pt\colon\cring{A}\bullet\rightarrow\mathbb R$ be a classical point, and let $\cringi{A}{1}{\bullet}$ be a minimal representative of $\cring{A}\bullet$ at $\pt$. Localizing $\cring{A}\bullet$, if necessary, we can assume that we have a surjective quasi-isomorphism $\cring{A}\bullet\rightarrow\cringi{A}{1}\bullet$. For any value of $t$ the subspace $\negsplit{E}{t}|_\pt\subseteq\tosplit{E}|_\pt$ has half the dimension and $\refo{\symfo}$ is negative definite on it. Thus it cannot contain any boundaries and, counting the Euler characteristic, we have $\dim\homdi{}(\negsplit{E}{t}|_\pt)=\dim\homdi{}(\negsplit{E}{1}|_\pt)$. We can choose just as many sections of $\negsplit{E}{t}$ around $\pt$, that are mapped injectively by $\homdi{}$, and this remains true for neighbouring values of $t$. Dividing by the dg submodule generated by these sections we clearly obtain a distribution that is a pull-back of a distribution over $\cringi{A}{1}\bullet$, which is a quotient by a graded mixed ideal.\end{proof}%

As with everything also the property of negative definiteness needs to be weakened, so as to become invariant with respect to weak equivalences.

\begin{definition} \label{DefNegDist}Let $(\dersch,\symstra)$ be a derived scheme with a $\mathbb C$-valued $-2$-shifted symplectic structure. Let $(\derm,\imfo{\symstra},\refo{\symstra})$ be the underlying dg manifold with $\mathbb R$-valued symplectic structures. An isotropic distribution 
	\begin{equation*}\anchor\colon(\lalge,\lagras)\longrightarrow(\tanga{\derm},\imfo{\symstra})\end{equation*}
is {\it negative definite with respect to $\refo{\symfo}$}, if $\refo{\symfo}$ defines a negative definite pairing on $H^1(\lalge|_\pt,\homdi{})$ for each $\pt\colon\cring{A}\bullet\rightarrow\mathbb R$.\end{definition}
Before making the gluing argument in the non-strict case we would like to show that negative definiteness on the cohomology can be made into negative definiteness on the fibers of $\tanga{\dera{A}}$ over closed points, if we switch to an equivalent integrable distribution.

\begin{proposition} Let $\dera{A}$ be a well presented dg $\mathbb C$-algebra and let $\symfo$ be a strict $-2$-shifted symplectic form in $\spec{\dera{A}}$. Given a semi-strict Lagrangian distribution
	\begin{equation*}(\lmin,\lagras)\longrightarrow(\tanga{\derm},\imfo{\symfo})\end{equation*}
that is negative definite with respect to $\refo{\symfo}$, there is an equivalent semi-strict Lagrangian distribution $(\lmin_1,\lagras_1)\rightarrow(\tanga{\derm},\imfo{\symfo})$, s.t.\@ $\forall\pt\colon\cring{A}\bullet\rightarrow\mathbb R$, $\refo{\symfo}$ is negative definite on $\nesplit{E_1}|_\pt$.\end{proposition}
\begin{proof} Let $\{e_j\}$ be a basis for $\nesplit{E}$ that is orthogonal with respect to $\refo{\symfo}$. Since $\symfo$ is non-degenerate at each $\pt\colon\dera{A}\rightarrow\mathbb C$ and $\imfo{\symfo}=0$ on $\nesplit{E}|_\pt=0$, it follows that $\forall j$ $\refo{\symfo}(e_j(\pt),e_j(\pt))\neq 0$. Every section of $\nesplit{E}$ that is a cycle over $\pt$ has to be a non-trivial cycle, hence if $\refo{\symfo}(e_j(\pt),e_j(\pt))>0$, $e_j$ cannot be a cycle over $\pt$. Let $k\in\mathbb N$ be s.t.\@ $j<k\Rightarrow\refo{\symfo}(e_j(\pt),e_j(\pt))<0$, and $j\geq k\Rightarrow\refo{\symfo}(e_j(\pt),e_j(\pt))>0$. Then $\{e_j\}_{j<k}\cup\{i\,e_j\}_{j\geq k}$ define an equivalent integrable distribution $\lmin_1$, s.t.\@ $\refo{\symfo}$ is negative definite on $\nesplit{E_1}$ over each $\pt$.\end{proof}%

In Prop.\@ \ref{TrivLa} we have seen that every purely derived foliation on $\spec{\cring{A}{\bullet}}$ having an isotropic structure relative to a strict $-2$-shifted symplectic form $\symfo$ can be equivalently presented as a semi-strict derived foliation, i.e.\@ $\symfo$ vanishes modulo $\homdi{}(\cring{A}{-1})$. Now we show that in the particular case of $\imfo{\symfo}$ and the isotropic distribution being in fact Lagrangian, there is an equivalent presentation, s.t.\@ $\imfo{\symfo}$ vanishes everywhere.

\begin{proposition} Let $\dera{A}$ be a well presented dg $\mathbb C$-algebra and let $\symfo$ be a strict $-2$-shifted symplectic form in $\spec{\dera{A}}$. Given a semi-strict Lagrangian distribution
	\begin{equation*}(\lmin,\lagras)\longrightarrow(\tanga{\derm},\imfo{\symfo})\end{equation*}
s.t.\@ $\refo{\symfo}$ is negative definite on $\lmin|_\pt$ for each classical $\pt$, there is an equivalent Lagrangian distribution $(\lmin',\lagras')\longrightarrow(\tanga{\derm},\imfo{\symfo})$ that is strict.\end{proposition}
\begin{proof} Since $\lmin$ is semi-strict relative to $\imfo{\symfo}$ we can choose a basis $\{e_j\}_{1\leq j\leq n}$ of $\lmin$ modulo $\homdi{}(\cring{A}{-1})$ that is orthonormal with respect to $-\refo{\symfo}$. Using Gram--Schmidt process, if needed, we can extend $\{e_j,ie_j\}_{1\leq j\leq n}$ to a basis of $\tosplit{E}$ over all of $\mani$. Let $\{\wt{e}_j\}_{1\leq j\leq n}$ be a basis of $\lmin$ over $\cring{A}{0}$. We have
	\begin{equation*}\forall j\quad\wt{e}_j=e_j+\underset{1\leq k\leq n}\sum(\homdi{}(\mu^k)e_k+\homdi{}(\nu^k) i e_k),\end{equation*}
where $\{\mu^k,\nu^k\}_{1\leq k\leq n}\subseteq\cring{A}{-1}$. Let $\phi\colon\nesplit{E}\rightarrow\tangde{\cring{A}\bullet}{0}$ be defined by 
	\begin{equation*}\forall j\quad \wt{e}_j\longmapsto\underset{1\leq k\leq n}\sum(\mu^k e_k+\nu^k i e_k).\end{equation*}
Then $\anchor-\homdi{}(\phi)$ is an equivalent integrable distribution that is strictly Lagrangian with respect to $\imfo{\symfo}$.\end{proof}%

If we try to glue strictly negative distributions defined on the same chart, we might encounter situations when two distributions are defined over the same open subset and they are equivalent. The natural question arises whether the process of interpolating between two \emph{equivalent} Lagrangian distributions take us out of this equivalence class. The following result shows that the answer is no.

\begin{proposition}\label{EquivalentInterpolating} Let $\nesplit{E_0},\nesplit{E_1}\subseteq\tosplit{E}$ be two maximally isotropic sub-bundles with respect to $\imfo{\symfo}$, s.t.\@ $\refo{\symfo}$ is negative definite on $\nesplit{E_0}$, $\nesplit{E_1}$, and suppose that the two distributions are equivalent. Then the family obtained by interpolating between these two (Prop.\@ \ref{StrictInterpolating}) lies within the same equivalence class of integrable distributions.\end{proposition}
\begin{proof} Since $\nesplit{E_0}$ and $\nesplit{E_1}$ are equivalent integrable distributions there is a weak equivalence $\phi$ over $\tanga{\cring{A}\bullet}$ from the first to a fibrant replacement of the second. As a fibrant replacement of $\nesplit{E_1}$ we can take an explicit construction we used in the proof of Prop.\@ \ref{TrivLa}. Then $\phi$ defines a map $\psi\colon\ubun\rightarrow i\nesplit{E_1}\cong\tosplit{E}/\nesplit{E_1}$. Consider a family of weak equivalences of integrable distributions: $\phi-t\homdi{}(\psi)$. This is exactly the family one obtains by interpolating as in Prop.\@ \ref{StrictInterpolating}.\end{proof}%

The last proposition has a very important corollary.

\begin{theorem} Let $(\dersch,\symfo)$ be a derived scheme with a strict $\mathbb C$-valued $-2$-shifted symplectic structure. Let $(\derm,\imfo{\symfo},\refo{\symfo})$ be the underlying dg manifold with $\mathbb R$-valued symplectic structures. The sheaf on $\mani$ of purely derived foliations that are Lagrangian distributions with respect to $\imfo{\symfo}$ and are negative definite with respect to $\refo{\symfo}$ is isomorphic to the sheaf of purely derived foliations that are strict Lagrangian distributions with respect to $\imfo{\symfo}$ and are strictly negative with respect to $\refo{\symfo}$.\end{theorem}
\begin{proof} A section of the sheaf of non-strict distributions is given by a compatible system of distributions on an atlas. Since locally every dg $\cinfty$-ring is cofibrant, such an atlas can be equivalently written as a compatible system of local distributions on $\derm$ itself. Then using a good coordinate system (e.g.\@ \cite{BoJ13}) we can apply Prop.\@ \ref{EquivalentInterpolating} to obtain one distribution on $\derm$.\end{proof}%

Now we have our main theorem as an immediate corollary.

\begin{theorem}\label{MainResult} Let $(\dersch,\symstra)$ be a derived scheme with a $\mathbb C$-valued $-2$-shifted symplectic structure. Let $(\derm,\imfo{\symstra},\refo{\symstra})$ be the underlying dg manifold with $\mathbb R$-valued symplectic structures. Suppose that the manifold $\mani$ of classical points is Hausdorff and second countable. Then the sheaf on $\mani$ of purely derived foliations that are Lagrangian distributions with respect to $\imfo{\symstra}$ and are negative definite with respect to $\refo{\symstra}$ is soft. In particular the set of global sections is not empty.\end{theorem}
\begin{proof} The fact that $\mani$ is Hausdorff and second countable implies that it is enough to show local softness of the sheaf. Therefore we can assume that $\symstra$ is strict. But then the sheaf of Lagrangian distributions that we want becomes isomorphic to the sheaf of strict Lagrangian distributions that are strictly negative definite. We know already that this sheaf is soft.\end{proof}%

\noindent{\small{\tt{dennis.borisov@uwindsor.ca, lkatzarkov@gmail.com, artan.sheshmani@gmail.com, yau@math.harvard.edu}}


\begin{thebibliography}{Bour}{\small
		\bibitem{BBBJ} O.Ben Bassat, Ch.Brav, V.Bussi, D.Joyce. {\it A 'Darboux theorem' for shifted symplectic structures on derived Artin stacks, with applications.} Geometry and Topology 19, pp.\@ 1287-1359 (2015).
		\bibitem{BoJ13} D.Borisov, D.Joyce. {\it Virtual fundamental classes for moduli spaces of sheaves on Calabi--Yau four-folds.} Geometry and Topology 21 (2017).
		\bibitem{BKS} D.Borisov, L.Katzarkov, A.Sheshmani. {\it Shifted symplectic structures on derived $Quot$-schemes.} Preprint.
		\bibitem{BSY2} D.Borisov, A.Sheshmani, S.-T.Yau. {\it Global shifted potentials for moduli stacks of sheaves on Calabi-Yau four-folds II.} arXiv:2007.13194 [math.AG]
		\bibitem{TopCha} D.Borisov. {\it Topological characterization of various types of $\cinfty$-rings.} Comm.\@ in Analysis and Geometry, vol.\@ 23, N.\@ 2, pp.\@ 349-361 (2015).
		\bibitem{BBJ13} Ch.Brav, V.Bussi, D.Joyce. {\it A Darboux theorem for derived schemes with shifted symplectic structure.} J.\@ of the AMS, vol.\@ 32 n.\@ 2, pp.\@ 399-443 (2018).
		\bibitem{CPTVV} D.Calaque, T.Pantev, B.To\"en, M.Vaqui\'e, G.Vezzosi. {\it Shifted Poisson structures and deformation quantization.} Journal of Topology 10, pp.\@ 483-584 (2017).
		\bibitem{CL} Y.Cao, N.C.Leung. {\it Orientability for gauge theories on Calabi--Yau manifolds.} Adv.\@ Math.\@ 314, pp.\@ 48-70 (2017).
		\bibitem{Dima} D.Carchedi, D.Roytenberg. {\it Homological algebra for superalgebras of differentiable functions.} arXiv:1212.3745v1.
		\bibitem{DerQuot} I.Ciocan-Fontanine, M.Kapranov. {\it Derived Quot schemes.} Ann.\@ Scient.\@ \'Ec.\@ Norm.\@ Sup.\@ 4th series, t.\@ 34 p.\@ 403-440 (2001).
		\bibitem{DT} S.K.Donaldson, R.P.Thomas. {\it Gauge theory in higher dimensions.}  In {\it The geometric universe} pp.\@ 31–47, OUP (1998).
		\bibitem{GR14} D.Gaitsgory, N.Rozenblyum. {\it Crystals and $D$-modules.} Pure and Applied Mathematics Quarterly, vol.\@ 10, no.\@ 1 (2014).
		\bibitem{Groth} A.Grothendieck. {\it Techniques de construction et th\'{e}or\`{e}mes d'existence en g\'{e}ometrie alg\'{e}brique IV. Les sch\'{e}mas de Hilbert.} S\'{e}minaire Bourbaki 221 (1960/61).
		\bibitem{Hinich} V.Hinich. {\it DG coalgebras as formal stacks.} J.\@ of Pure and Applied Algebra 162, pp.\@ 209-250 (2001).
		\bibitem{HuyLehn} D.Huybrechts, M.Lehn. {\it The geometry of moduli spaces of sheaves.} Cambridge University Press (2010).
		\bibitem{Conan} N.C.Leung. {\it Riemannian geometry over different normed division algebras.} J.\@ Differential Geometry 61, pp.\@ 289-333 (2002).
		\bibitem{P14} T.Pantev. {\it Derived foliations and shifted potentials.} Talk at ``Algebra, geometry and physics: a conference in honour of Maxim Kontsevich'', IHES 23-27.06.2014.
		\bibitem{PTVV13} T.Pantev, B.To\"en, M.Vaqui\'e, G.Vezzosi. {\it Shifted symplectic structures.} Publ. Math. Inst. Hautes \'Etudes Sci. 117, pp. 271-328 (2013).
		\bibitem{ToVa} B.To\"en, M.Vaqui\'e. {\it Moduli of objects in dg-categories.} Ann.\@ Scient.\@ \'Ec.\@ Norm.\@ Sup.\@ 4th series, t.\@ 40, p.\@ 387-444 (2007). 
		\bibitem{To10} B. To\"{e}n. {\it Simplicial presheaves and derived algebraic geometry.} pp. 119-186 in I.Moerdijk, B.To\"{e}n. {\it Simplicial methods for operads and algebraic geometry.} Birkh\"{a}user (2010).
		\bibitem{TV11} B.To\"en, G.Vezzosi. {\it Alg\`ebres simpliciales $S^1$-\'equivariantes, th\'eorie de de Rham et th\'eor\`emes HKR multiplikatifs.} Compositio Math.\@ 147 pp.\@ 1979-2000 (2011).
		\bibitem{AlgF1} B.To\"en, G.Vezzosi. {\it Algebraic foliations and derived geometry I: the Riemann--Hilbert correspondence.} arXiv:2001.05450 [math.AG]
		\bibitem{AlgF2} B.To\"en, G.Vezzosi. {\it Algebraic foliations and derived geometry II: the Grothendieck-Riemann-Roch theorem.} arXiv:2007.09251 [math.AG]
		\bibitem{UY86} K. Uhlenbeck, S.T. Yau. {\it On the existence of {H}ermitian-{Y}ang-{M}ills connections in stable vector bundles.} Comm. Pure Appl. Math. \@ 257-293\@ 39 (1986).
		\bibitem{Vaintrob} A.Yu.Vaintrob. {\it Lie algebroids and homological vector fields.} Russ.\@ Math.\@ Suv.\@ 52 428 (1997).
		
	}
	
	
\end{thebibliography}
\end{document}

